\documentclass[a4paper,12pt,oneside]{amsart}

\usepackage{amsmath,amsfonts,amssymb,amsthm,amsopn,amsxtra}

\newtheorem*{Thm}{Theorem}    
\newtheorem*{Cor}{Corollary}
\newtheorem{Cor1}{Corollary}
\newtheorem*{ThmA}{Theorem A}
\newtheorem{Lem}{Lemma}
\newtheorem{Pro}{Proposition}

\theoremstyle{remark}
\newtheorem*{Rems}{Remarks}
\newtheorem*{Ex}{Example}

\DeclareMathOperator{\Aut}{Aut}
\DeclareMathOperator{\Int}{Int}

\DeclareMathOperator{\im}{im}
\DeclareMathOperator{\dc}{dc}
\DeclareMathOperator{\pr}{pr}
\DeclareMathOperator{\tr}{tr}

\newcommand{\Cal}{\mathcal}
\newcommand{\fr}{\mathfrak}
\newcommand{\C}{\mathbb{C}}        
\newcommand{\R}{\mathbb{R}}        
\newcommand{\Rp}{\mathbb{R}^+}     
\newcommand{\Z}{\mathbb{Z}}        %
\newcommand{\N}{\mathbb{N}}        %
\newcommand{\T}{\mathbb{T}}        %
\newcommand{\He}{\mathbb{H}}       
\newcommand{\G}{\mathbb{G}}        
\newcommand{\SO}{\mathit{SO}}
\newcommand{\SU}{\mathit{SU}}
\newcommand{\VN}{\mathit{VN}}
\newcommand{\bW}{\overline{\Cal W}}
\newcommand{\bbW}{\overline{\overline{\Cal W}}}
\newcommand{\bnu}{{\boldsymbol \nu}}
\newcommand{\brho}{{\boldsymbol \rho}}
\newcommand{\btheta}{{\boldsymbol \theta}}
\newcommand{\bx}{{\mathbf x}}

\begin{document}
\title[weak amenability]{On weak amenability of Fourier algebras}
\author{Viktor  Losert}
\address{Institut f\"ur Mathematik, Universit\"at Wien, Strudlhofg.\ 4,
  A 1090 Wien, Austria}
\email{Viktor.Losert@UNIVIE.AC.AT}
\date{29 November 2019}


\maketitle

\baselineskip=1.3\normalbaselineskip
\section{Introduction and Main results}\vspace{1mm} \label{Intro}
Let $A$ be a Banach algebra, $X$ a Banach$\ A$--bimodule. A linear map
$D\!:A\to X$ is called a derivation if \ $D(ab)=aD(b)+D(a)\,b$ \,for
$a,b\in A$\,. By $A'$ we denote the dual space of $A$ with the dual action
(also called the dual module of $A$).
$A$ is called {\it weakly amenable} if every bounded derivation $D\!:A\to A'$
is inner \ (i.e. there exists $f\in A'$ such that $D(a)=af-fa$ \,for $a\in A$).

From now, we assume that $A$ is commutative.
$X$ is called a symmetric\linebreak $A$--bimodule if $ax=xa$ for
$a\in A\,,\,x\in X$\,.
Then (\cite{D}\,Th.\,2.8.63) $A$ is weakly amenable iff there is no
non--zero bounded derivation from $A$
to any symmetric Banach $A$--bimodule $X$\,.\vspace{2mm}

Our main subject will be $A=A(G)$ \, the Fourier algebra of a locally compact
group $G$\,.\vspace{-2mm}
\begin{Thm}
Let $G$ be a connected locally compact group. If
$A(G)$ is weakly amenable then $G$ is abelian. 
\end{Thm}
\begin{Cor}
Let $G$ be a locally compact group with identity component $G^0$.\\
$A(G)$ is weakly amenable if and only if $G^0$ is abelian.\vspace{1mm}
\end{Cor}\noindent
{\bf Some History} \\
The notion of weak amenability was introduced by
Bade, Curtis and Dales for commutative Banach algebras and then extended by
Johnson, see \cite{D}\,p.\,306 for more details and references. For the
Fourier algebra it was shown in \cite{J} that $A(G)$ is not weakly amenable
for $G=\SO(3)$ and then in \cite{Pl} for $G$ any non-abelian compact connected
Lie group. For non-compact groups, the case of the $ax+b$ group was settled
in \cite{CG} and then the Heisenberg group in \cite{CG2}. Finally in
\cite{LLSS} it was shown that $A(G)$ is not weakly amenable for any
non-abelian connected Lie group $G$\,.
\begin{proof}[Proof of the Corollary]
If $G^0$ is abelian it was shown by Forrest and Runde (\cite{FR} Th.\,3.3)
that $A(G)$ is weakly amenable. There is a heredity for subgroups
(\cite{FR} L.\,3.1): \ if $G$ is any locally compact
group and $A(G)$  is weakly amenable then $A(H)$ is weakly
amenable for every closed subgroup $H$ of $G$\,. Hence the converse in
the Corollary follows from the Theorem (and this was already conjectured
in \cite{FR}).
\end{proof}
\begin{proof}[Outline of Proof of the Theorem]
The proof of the theorem will take most of the remaining sections. The
basic strategy follows the approach of the earlier authors. By the
heredity for subgroups (see above) it will be enough to prove the theorem
for ``small'' non--abelian connected groups. In Section\,\ref{mini} we
present lists \eqref{minsc}$,$\eqref{minco} of the minimal non--abelian
connected groups (extending that in \cite{LLSS}). Using the results
from Appendix\,\ref{dense} on connected dense subgroups we show
(Proposition\,\ref{minipro}) that every non--abelian connected locally
compact group has a closed subgroup
isomorphic to one of the groups in \eqref{minsc}$\,\cup\,$\eqref{minco}.
\\
For the treatment of $A(G)$ for specific groups, \cite{LLSS} used techniques
from spectral synthesis. Put $\check\Delta_G=\{(g,g^{-1}):g\in G\}$ (the
anti--diagonal). For a connected Lie group $G$ they show
(\cite{LLSS}\,Th.\,3.2) that $A(G)$ is weakly amenable iff $\check\Delta_G$
is of local synthesis for $G\times G$\,. This synthesis property is inherited
by locally isomorphic groups, hence it was enough to check the simply
connected groups from \eqref{minsc}. But, so far, this approach works
only for Lie groups.
\\
We will use the method of Johnson, Choi, Ghandehari who consider
explicit derivations arising from differentiation operators.
Differentiation operators are prototypes of derivations on function spaces
and their continuity properties express some kind of smoothness for the
functions. It is a basic construction for a Lie group $G$ with Lie algebra
$\fr g$ that the elements of $\fr g$ correspond to point derivations (at $e$)
on the algebra $C^\infty(G)$ of $C^\infty$-functions.
These point derivations extend to left invariant (or right invariant)
derivations on $C^\infty(G)$ (\cite{V}\,Sec.\,2.3 and 1.1). Norm estimates for
such derivations are done using Fourier transforms. In
Appendix\,\ref{AppenB} we review the relevant properties of the
non--commutative Fourier transform and also the applications to Fourier
algebras. It turns out that in some cases the derivations that we consider
do not map to the dual $\VN(G)$ of $A(G)$. Then (taking up constructions
from \cite{CG2}) we consider appropriate
$A(G)$--modules which needs some knowledge about the dual convolution for
these groups. In the case of the group $M(2)$ (and the related groups
$\widetilde{M(2)}$ and $M(2)_{\Cal K}$) this becomes somewhat more involved. 
\end{proof}\vspace*{1mm plus 1cm}\noindent
{\bf Notations} \\
On semidirect products $G=H\ltimes K$\,: \;in those examples where the
notation should (hopefully) make it easy to distinguish elements of $H$ and
$K$ we will use the internal presentation (see \cite{HR}\,I.2.6). Thus
$H$ and $K$ are considered as (closed) subgroups of $G$\,, with $H$ normal,
$H\cap K=\{e\} \,(e$ denoting the unit element of $G$\,),\ G=HK
\,(\,$(h,k)\mapsto hk$ defining a homeomorphism $H\times K\to G$\,). There
is given a continuous (left) action of $K$ on $H$\,, denoted by $\circ$\,,
(equivalently, a continuous homomorphism $K\to\Aut(H)$\,) and the group
multiplication of $G$ is governed by the rule $kh=(k\circ h)\,k$ \,for
$k\in K,\,h\in H$\,. Then a mapping $\varphi\!:G\to G_1\ (G_1$ some other
topological group) is a continuous homomorphism iff its restrictions to
$H$ and $K$ are continuous homomorphisms and it preserves the action,
i.e. $\varphi(k\circ h)=\varphi(k)\varphi(h)\varphi(k)^{-1}$. In
Section\,\ref{axb} we use a slightly modified notation.
\section{Minimal non--abelian connected groups}\vspace{1mm} \label{mini}
In \cite{LLSS}\,Prop.\,3.1 the minimal non--abelian real Lie algebras are
listed. The corresponding simply connected (and connected) Lie groups
are
\begin{equation}\tag{$\Cal M$}\label{minsc}
\SU(2)\,,\qquad \R^n\rtimes\R\quad (n\in\{1,2\})
\end{equation}
with the actions ($t,x,y,z\in\R\,,\; v\in\R^2$)
\begin{flalign*}
n=1\quad &t\circ x=e^tx &&ax+b \text{ \it group} \\
n=2\quad &t\circ(y,z)=(y,z+ty) &&\He \text{ \it Heisenberg group}\\
 &t\circ v=A_tv\ \text{ with } A_t=\begin{pmatrix}\cos t &-\sin t\\
\sin t & \phantom{-}\cos t
\end{pmatrix}&&\widetilde{M(2)}\text{ \it universal covering of the}\\[-2.5mm]
\intertext{\hfill\it Euclidian motion group of the plane}\\[-9mm]
 &t\circ v=A_tv\ \text{ with } A_t=e^{at}(\begin{smallmatrix}\cos t &-\sin t\\
\sin t & \phantom{-}\cos t
\end{smallmatrix})&&\G_a\text{ \it Gr\'elaud groups },\;a>0\,.
\end{flalign*}\noindent
Their Lie algebras do not have a proper non--abelian subalgebra and any
non--abelian finite dimensional real Lie algebra contains a copy of one
of them.
\begin{Lem}   \label{minile}
Let $G$ be a non--abelian connected locally compact group. Then there exists
a group $H$ from \eqref{minsc} and a continuous homomorphism
$\varphi\!:H\to G$ with discrete kernel.
\end{Lem}
\begin{proof}
We represent $G$ using an Iwasawa pair, $G\cong (L\times K)/D$ where
$L$ is a connected Lie group, $K$ a compact group (not necessarily connected),
$D$ a discrete central subgroup (\cite{G1}\,Prop.\,5, see also
\cite{PW}\,sec.\,2 for more properties and \cite{L}\,p.\,120 for further
references).  This gives
homomorphisms $j_L\!\!:L\to G\,,\ j_K\!\!:K\to G$ with discrete kernel
(in fact, it is easy to modify the representation so that $j_L\,,\,j_K$ are
injective). $G$ being connected, the image of $L\times K^0$ must be dense.
Hence if $G$ is non--abelian, then either $L$ or $K^0$ must be non--abelian.
If $L$ is non--abelian there exists (using e.g.\,\cite{V}\,Th.\,2.7.5 and
Cor.\,2.7.4) a group $H$ from \eqref{minsc} and a
homomorphism $\varphi_0\!\!:H\to L$ with discrete kernel. Then we may take
$\varphi=j_L\circ\varphi_0$\,. Similarly, if $K^0$ is non--abelian then
(using e.g.\,\cite{HM}\,Cor.\,9.6,\,Th.\,9.19(i),(ii) and Prop.\,6.46(IV)) it
follows that there exists a homomorphism $\varphi_0\!\!:\SU(2)\to K$ with
discrete kernel.
\end{proof}
We see that for arbitrary connected groups one has to consider also
closures of non--abelian homomorphic images of the groups from \eqref{minsc}.
This gives in addition the groups
\begin{equation}\tag{$\Cal M+$}\label{minco}
\SO(3)\,,\qquad
\He_{\Cal K}\;,\qquad 
M(2)_{\Cal K} 
\end{equation}
where $\Cal K$ is a (non--zero)
solenoidal group \;i.e. 
$\Cal K$ compact and there exists a continuous homomorphism
$\varphi_0\!\!:\R\to \Cal K$ with dense image, $\Cal K$ is written additively.
\\[2mm]
$\He_{\Cal K}=(\R\times\Cal K)\rtimes\R$\,, \ action
$t\circ(y,k)=(y,k+\varphi_0(ty))\quad(t,y\in \R\,,\,k\in\Cal K$\,).
\\[2mm]
$M(2)_{\Cal K}=\R^2\rtimes\Cal K$ \ with the additional assumption that there
exists a proper closed subgroup $\Cal K_0$ of  $\Cal K$ such that
$\varphi_0(2\pi\Z)=\Cal K_0\cap\varphi_0(\R)$\,. Then
$\Cal K/\Cal K_0\cong\R/2\pi\Z$ and the action of $\R$ on $\R^2$ (defining
$\widetilde{M(2)}$\,) induces an action of $\Cal K$ on $\R^2$.
\\[2mm]
Special cases: \
For $\Cal K=\R/\Z$ \,($\varphi_0$ the quotient mapping) \;$\He_{\Cal K}$ is
the {\it reduced Heisenberg group}. \
For $\Cal K=\R/2\pi\Z$ \,($\Cal K_0=\{0\}$) \;$M(2)_{\Cal K}$ equals $M(2)$
the {\it Euclidian motion group of the plane}.
\begin{Pro} \label{minipro}
Every non--abelian connected locally compact group $G$ has a closed subgroup
isomorphic to one of the groups in \eqref{minsc}$\,\cup\,$\eqref{minco}.
\end{Pro}
\begin{proof} By Lemma\,\ref{minile} there exists a group $H$ from
\eqref{minsc} and a continuous homomorphism $\varphi\!:H\to G$ with discrete
(hence central) kernel. Replacing $G$ by $\overline{\varphi(H)}$ we want to
show that $G$ is isomorphic to one of the groups in
\eqref{minsc}$\,\cup\,$\eqref{minco}.\\
For $H=\SU(2)$, $\varphi(H)$ is closed, hence $G\cong H/\ker\varphi$
(\cite{HR}\,Th.\,5.29 and Th.\,5.27). Since $Z(\SU(2))=\{\pm I\}$ (centre),
it follows
that $G\cong\SU(2)$ or $G\cong\SU(2)/\{\pm I\}$. It is classical (see e.g.
\cite{HM}\,Ex.\,E1.2(ii)\,) that $\SU(2)/\{\pm I\}\cong\SO(3)$.
\\
For the other groups in \eqref{minsc} we use the results and constructions
from Appendix\,\ref{dense}. For the $ax+b$--group and the Gr\'elaud groups
one can check that the centre is trivial, hence $\varphi$ must be injective.
A necessary condition for $x\in H$ to generate a relatively compact subgroup
in $\Aut(H)$ is relative compactness of the orbit $\{x^nyx^{-n}:n\in\Z\}$
for every $y\in H$\,. Again one can check that this does  not happen for
$x\neq e$\,. In the notation of Theorem\,\ref{dense} it follows that the
vector group $V$ is trivial, hence $\varphi(H)$ is always closed and
as above it follows that $\varphi(H)\cong H$ \,(i.e. these are examples
of ``absolutely closed'' groups, \cite{G2}\,p.\,728).
\\
For $H=\He$ the Heisenberg group an element generates a relatively compact
subgroup in $\Aut(H)$ iff it belongs to the centre $Z(H)=\{(0,z):z\in\R\}$
(notation as in \eqref{minsc}). Hence the only non--trivial vector group as in
Theorem\,\ref{dense} is $V=Z(H)$\,. Then $\ker\varphi$ is trivial and
the action $\alpha$ is trivial on $\varphi(V)$\,. Thus the action $\alpha_0$
in Proposition\,\ref{densepr} (one has $H_1=H$) is trivial on
$\Cal K=\overline{\varphi(V)}$\,.
By Proposition\,\ref{densepr}(b) it follows that $G\cong (H\times\Cal K)/D$\,.
It is easy to see that the homomorphism
$((y,z)t,k)\mapsto(y,\varphi_0(z)+k)\,t$
\,(mapping $H\times\Cal K$ onto $\He_{\Cal K}$) induces an isomorphism of
$(H\times\Cal K)/D$ and $\He_{\Cal K}$\,. These gives the cases where
$\ker\varphi_0=\ker\varphi$ is trivial. But the construction of
Proposition\,\ref{densepr} works more generally for $\Cal K$ as in
\eqref{minco} with $\ker\varphi_0$
discrete (see also Remark\,(c) in Appendix\,\ref{dense}). In
Theorem\,\ref{dense} this adds the cases where $V$ is trivial. Then
$\varphi$ is surjective, hence as above $G\cong H/\ker\varphi$\,. When
$\ker\varphi_0$ is a non--trivial subgroup of $Z(H)\cong\R$, $\varphi_0$ is
also surjective, thus
$\Cal K\cong\R/\ker\varphi_0$\,. It gives the reduced Heisenberg group
and some variants that are isomorphic to it.
\\[3mm plus 1mm]
For $H=\widetilde{M(2)}=\R^2\rtimes\R$ the universal covering of the
Euclidian motion group of the plane, $h\in H$ generates a relatively compact
subgroup in $\Aut(H)$ iff $h\notin\R^2$ and then $h$ is conjugate to an
element of $\R$\,. Hence, in Theorem\,\ref{dense}, if $V$ is non--trivial
we may assume that $V=\R$ (the second factor in the semidirect product).
The centre $Z(H)$ is $2\pi\Z$ (subgroup of $V$), thus $\ker\varphi$ is
trivial. Put $\Cal K=\overline{\varphi(V)},\
\Cal K_0=\overline{\varphi(2\pi\Z)},\ \varphi_0=\varphi\vert V$.
By Proposition\,\ref{densepr}(b) (splitting case, $H_2=\R^2$),
$G\cong \R^2\rtimes\Cal K$\,.
The action $\alpha$ is trivial on $\Cal K_0$ and non--trivial on
$\varphi(\R\setminus2\pi\Z)$ \,(using that $\alpha\circ\varphi=\iota_H$).
It follows that $\varphi_0(2\pi\Z)=\Cal K_0\cap\varphi_0(\R)$ and one can
check that the action used in Proposition\,\ref{densepr}(b) (restricting
$\alpha_0$ to $\R^2$) coincides with that described in \eqref{minco}.
Again there are the additional cases where $V$ is trivial. Then
$\varphi$ is surjective,
$G\cong H/\ker\varphi\cong\R^2\rtimes(\R/\ker\varphi)$\,.
$\ker\varphi$ has to be a subgroup of $2\pi\Z$\,. Thus
$\Cal K\cong \R/2k\pi\Z\ \ k=0,1,\dots$\,, giving $M(2)\ \;(k=1)$ and its
finite covering groups $(k>0)$, and of course $\widetilde{M(2)}\ \;(k=0)$.
Conversely, if $\Cal K$ is a compact group as in \eqref{minco} it is easy
to see that $\varphi_0$ extends to a homomorphism
$\varphi\!:\widetilde{M(2)}\to M(2)_{\Cal K}$ with dense image and discrete
kernel.\vspace{-2mm plus 3mm}
\end{proof}
\begin{Rems}
(a) For an alternative argument, one can compute the adjoint action of
the groups. It turns out that all the groups $H$ in \eqref{minsc} are
$CA$--groups, i.e. $\Int(H)$ is closed in $\Aut(H)$. By \cite{G2}\,4.1,
if the connected Lie group $H$ is a $CA$ group, $\varphi\!:H\to G$ an
injective (continuous) homomorphism with dense image, $G$ a Hausdorff
topological group, then
$G=\varphi(H)\,\overline{\varphi(Z)}$ \,($Z=Z(H)$ denoting the centre of $H$).
This holds more generally when $H$ is a connected locally compact group
and a $CA$ group, $\ker\varphi$ totally disconnected
(compare \cite{PW}\,Cor.\,4.2). It follows that $G$ is a central extension
of $H/Z$\,.\vspace{1mm}
\item[(b)] As explained in Appendix\,\ref{dense},\,Rem.\,(c), for a
(necessarily abelian)
compact group $\Cal K$ as in \eqref{minco}, the dual group $\widehat{\Cal K}$
is isomorphic to a subgroup of $\R$\,. For the Heisenberg group $\He$ any
non--zero subgroup $\Gamma$ of $\R$ defines in this way a group $\He_{\Cal K}$
(the trivial subgroup would give $\R^2\cong\He/Z(\He)$, this is a case
where $\ker\varphi$ is\linebreak non--discrete, as mentioned in
Appendix\,\ref{dense},\,Rem.\,(c); it follows that there is no minimal
extension $\He_{\Cal K}$). One can show that
two subgroups $\Gamma,\Gamma'$ define isomorphic groups
$\He_{\Cal K},\He_{\Cal K'}$ iff $\Gamma'=t\Gamma$ for some $t>0$.
\\
For $H=\widetilde{M(2)}$ a restriction comes through the adjoint action.
Put $\chi_t(x)=e^{itx}$ and use $\chi_t\leftrightarrow t$ for the
identification of $\widehat\R$ and $\R$\,. One computes that (in the
notation of Appendix\,\ref{dense},\,Rem.\,(c))
\;$\widehat\iota_0(\Gamma_0)=\Z$\,, hence $\widehat\varphi_0(\widehat{\Cal K})$
must be a subgroup of $\R$ containing $\Z$ \,and this is also equivalent
to the condition $\varphi_0(2\pi\Z)=\Cal K_0\cap\varphi_0(\R)$ of
\eqref{minsc} \,(and then
$\widehat{\Cal K_0}\cong\widehat\varphi_0(\widehat{\Cal K})/\Z $).
Different groups $\Gamma,\Gamma'\subseteq\R$ (containing $\Z$) give
non--isomorphic groups $M(2)_{\Cal K}\,,M(2)_{\Cal K'}$\,.
The minimal case is $\Gamma=\Z$ with
$\Cal K=\R/2\pi\Z\,,\ M(2)_{\Cal K}=M(2)$\,. The maximal case is $\Gamma=\R$
(discrete topology) with $\Cal K=b(\R)$ (Bohr compactification), this
is also mentioned in \cite{LLSS}\,Ex.\,3.4 (denoted as $\R^{ap}$).
$M(2)_{\Cal K}$ is a Lie group iff $\Gamma$ is finitely generated and it
is second countable iff $\Gamma$ is countable.
\end{Rems}

\section{The groups $\SO(3)$ and $\SU(2)$}\vspace{1mm} \label{comp}
For $G=\SU(2)$ Johnson's method has been elaborated in \cite{CG}\,sec.\,3.
We add here the aspect of (generalized) convolution operators that
will be useful in the following sections.
\\
The Lie algebra $\fr g$ of $G$ consists of the complex skew-hermitian
$2\times2$ matrices with trace $0$\,. For $X\in\fr g$ we consider
the distribution
$d=\lim_{t\to0}\frac1t(\delta_{\exp(-tX)}-\delta_I)$ (directional derivative
for the tangent vector $-X$\,; $\delta_x$ point measure at $x$\,, $I$
identity matrix). We put $D(f)=f\star d$\,. Classically, $d$ and $D$ are
defined for $C^1$-functions, $d$ gives a point derivation at the unit $I$
of $G$\,, $D$ defines a left invariant derivation on $C^\infty(G)$ (i.e.
$D(\delta_x\star f)=\delta_x\star D(f)$ for all $x\in G$). The example
of \cite{CG} (related to that of Johnson \cite{J} for $\SO(3)$) is
obtained for $X_0=\left(\begin{smallmatrix}
\,\frac i2 & \;0 \\[.5mm]
\,0 & -\frac i2 
\end{smallmatrix}\right)$.
\\
We take normalized Haar measure on $G$\,, then the Plancherel measure on the
discrete space $\widehat G$ gives the weights
$\nu(\{\pi\})=\dim(\pi)$\,. Thus
for $f\in A(G)$ ($\subseteq L^1(G)$ in the compact case)
$\lVert f\rVert_A=
\sum_{\pi\in\widehat G}\dim(\pi)\lVert\pi(f)\rVert_1$
and $\lVert f\rVert_{\VN}=\sup_{\pi\in\widehat G}\lVert\pi(f)\rVert$
for $f\in L^1(G)$. If $f$ is a $C^1$-function such that
$f,f\star d\in L^1(G)$ it is easy to see that
$\pi(f\star d)=\pi(f)\circ\pi(d)$ (using the extension of $\pi$ to compactly
supported distributions). For $X_0$ as above, $\pi(d)=-F_\pi$ (for the
standard realization of $\pi$) as given in \cite{CG}\,p.\,6508. As explained
there, one has $\lVert F_\pi\rVert=\frac{\dim(\pi)-1}2$.
It follows that $\lVert f\star d\rVert_{\VN}\le\frac12\lVert f\rVert_A$\,.
By \cite{E}\,Prop.\,3.26 $A(G)$ contains the space $\Cal D(G)$ of
compactly supported $C^\infty$-functions as a dense subspace (for any
Lie group $G$). It follows that $D$ extends to a bounded operator
$A(G)\to\VN(G)$ which is still a derivation.
\\
This construction of derivations on $C^\infty(G)$ can be applied on any
Lie group $G$\,. For a compact semisimple Lie group, similar arguments
as above show that this gives derivations $A(G)\to\VN(G)$ for any
$X\in\fr g$\,. But the same estimate applies also to the corresponding
right invariant derivations $D^{(r)}f=d\star f$\,. Combined the method
produces a subspace of derivations $A(G)\to\VN(G)$ with the dimension
$2\dim(G)$ \,(for any compact semisimple Lie group, with the extended theory
of \cite{HM}\,Th.\,9.49 this extends to some extent
even to general compact connected semisimple groups).
In any case the guess made in \cite{LLSS}\,Rem.\,2.4 needs some revision.
\vspace{2mm plus 2mm}

\section{The \,$ax+b$ \,group}\vspace{1mm} \label{axb}
$G=\R\rtimes\R$\,, the (proper) affine transformations of the real line.
For clarity, we use the notation $G=\beta(\R)\,\alpha(\R)$, i.e.
$(x,t)\mapsto\beta(x)\,\alpha(t)$ defines a homeomorphism $\R^2\to G\,,\
\,\alpha,\beta:\R\to G$ are homomorphisms and
group multiplication is governed by the relation
$\alpha(t)\,\beta(x)=\beta(e^tx)\,\alpha(t)$ \,($t,x\in\R$;
hence $\beta(\R)$ is the normal factor). For (left) Haar measure we take
$e^{-t}dx\,dt$\,. $G$ is not unimodular, $\Delta(\beta(x)\,\alpha(t))=e^{-t}$.
Irreducible representations are described in numerous texts (and with
numerous variations), starting with Gelfand, Naimark. We keep close to
\cite{CG}\,sec.\,4. Take $\Cal H=L^2(\Rp,\frac1\psi d\psi)$ (with
$\Rp=[0,\infty[$) and
for ${f\!\in\!\Cal H\,,}\ (\pi^+(\beta(x))f)(\psi)=e^{-ib\psi}f(\psi),\
(\pi^-(\beta(x))f)(\psi)=e^{ib\psi}f(\psi),\
(\pi^\pm(\alpha(t))f)(\psi)=f(e^t\psi),\ (Kf)(\psi)=\psi f(\psi)\
\;(t,x\in\R\,,\,\psi\in\Rp)$.
$\pi^+,\pi^-$ define non-equivalent representations and up to equivalence
these are the only infinite dimensional
irreducible representations of~$G$\,. $K\ (=K_+,K_-)$ is a Duflo--Moore
operator for $\pi^\pm$, the Plancherel measure is concentrated on
$\{\pi^+,\pi^-\}$ with $\nu(\{\pi^+\})=\nu(\{\pi^-\})=\frac1{2\pi}$\,.
Thus $\Cal F$ maps $\VN(G)$ isomorphically to the $l^\infty$-sum
$\Cal B(\Cal H)\oplus\Cal B(\Cal H)$. $\Cal F_A(f)=([\pi^\pm(f)K]),\
\lVert f\rVert_A=
\frac1{2\pi}(\lVert[\pi^+(f)K]\rVert_1+\lVert[\pi^-(f)K]\rVert_1)$ for
$\;f\in A(G)\cap L^1(G)$. $\Cal F_A$ extends to an isometry of $A(G)$ onto
the $l^1$-sum $\Cal S^1(\Cal H)\oplus\Cal S^1(\Cal H)$ weighted by~$\nu$\,.
\\[2mm plus 2mm]
We consider the distribution
$d=\lim_{x\to0}\frac1x(\delta_{\beta(-x)}-\delta_{\beta(0)})$ and
$D(f)=f\star d$  \;(taking
a direction from the normal factor; up to a constant factor this is
the derivation $D_\flat$ considered in \cite{CG}\,p.\,6510). It follows that
$(\pi^+(d)g)(\psi)=i\psi g(\psi),\ (\pi^-(d)g)(\psi)=-i\psi g(\psi)$ hence
$\pi^\pm(d)=\pm iK$\,. If $f$ is a $C^1$-function on $G$ such that
$f,f\star d\in L^1(G)$ one has $\pi(f\star d)=[\pi(f)\circ\pi(d)]$.
Hence, if in addition $f\in A(G)$ holds, it follows that
$\lVert f\star d\rVert_{\VN}\le2\pi\lVert f\rVert_A$\,,
$D$ extends to a bounded derivation $A(G)\to\VN(G)$\,.
\\[1mm]
If $d$ involves a direction from the non-normal factor one can show that
this does not give a bounded mapping $A(G)\to\VN(G)$\,, the same when
considering right invariant derivations $D^{(r)}f=d\star f$\,. Hence for
the $ax+b$ \,group the method produces just a one dimensional space of
derivations $A(G)\to\VN(G)$.
\vspace{1.2mm plus 2mm}
\section{The Heisenberg group $\He$}\vspace{1mm} \label{heis}
We use the parametrization given in Section\,\ref{mini},
$\He=\R^2\rtimes\R$\,, considering $\R,\R^2$ as subgroups (i.e. omitting
embeddings $\alpha,\beta$ as in Section\,\ref{axb}) elements are written
as $v\,x$ with $x\in\R\,,\;v\in\R^2$, with multiplication rule
$x\,v=(x\circ v)\,x$ where $x\circ v=(y,z+xy)$ for $v=(y,z)$\,. For Haar
measure we take standard Lebesgue measure of $\R^3$. $\He$ is unimodular.
The representation
theory is again classical. Take $\Cal H=L^2(\R)$. For real $t\neq0$ an
irreducible representation $\pi_t$ is defined for $f\in\Cal H$ by
$(\pi_t(v)f)(\psi)=e^{it(z-y\psi)}f(\psi)\,,\ (\pi_t(x)f)(\psi)=f(\psi-x)\
\;(x,\psi\in\R\,,\;v=(y,z)\in\R^2)$\,. The representations $\pi_t$ are
pairwise non-equivalent and exhaust (up to equivalence) all the infinite
dimensional irreducible representations of $\He$\,. The image of
$\{\pi_t:t\neq0\}$ (with the topology inherited from $\R\setminus\{0\}$)
defines a locally compact (but non-closed) subset of $\widehat\He$\,.
The Plancherel measure is concentrated on this subset and with the given
parametrization of the representations we have
$d\nu=\frac1{4\pi^2}\lvert t\rvert dt$\,. As explained in
Appendix\,\ref{AppenB}, it follows that
$\Cal B_1^\oplus(\widehat\He,\nu)=L^1(\widehat\He,\nu,\Cal S_1(\Cal H))\cong
L^1(\R,\frac1{4\pi^2}\lvert t\rvert dt,\Cal S_1(\Cal H))$
and
$\lVert f\rVert_A=
\frac1{4\pi^2}\int_\R\lVert\pi_t(f)\rVert_1\lvert t\rvert\,dt$
for $f\in L^1(\He)\cap A(\He)$\,. General elements of
$\Cal B_1^\oplus(\widehat\He,\nu)$ will be denoted as $T=(T_t)_{t\in\R}$\,.
\\[0mm plus 2mm]
We review now the approach of \cite{CG2} for $\He$\,. $Z=\{(0,z):z\in\R\}$
is the centre of $\He$ and we consider the distribution
$d=\lim_{z\to0}\frac1z(\delta_{(0,-z)}-\delta_{(0,0)})$ and
$D(f)=f\star d$ \;(up to a constant factor this is
the derivation $\partial_Z$ considered in \cite{CG2}\,(4.1)\,).
It follows that $(\pi_t(d)g)(\psi)=-itg(\psi)$ (for $g\in\Cal H$) and
it is easy to see that now $D$ does {\it not} produce a bounded mapping
$A(\He)\to\VN(\He)$\,. Using
distributional derivatives we have a linear mapping
$D\!:A(\He)\to\Cal D(\He)'$. The idea is to
construct a symmetric Banach $A(G)$--bimodule containing the range
of $D$ and making $D$ bounded. The module operation should extend pointwise
multiplication. The obstacle is that the product with functions is not
defined for general distributions and we should show that the distributions
in the image of $D$ are sufficiently tame to allow multiplication with
functions from $A(\He)$. Technically, this becomes practicable when working
on the dual side.
\\[0mm plus 1mm]
Put $\Cal W=L^1(\R,dt,\Cal K(\Cal H))$ and for
$T=(T_t)\in\Cal B_1^\oplus(\widehat\He,\nu)$ we put
$\widehat D(T)=(-itT_t)$. Obviously,
$\widehat D\!:\Cal B_1^\oplus(\widehat\He,\nu)\to\Cal W$ and
$\lVert\widehat D(T)\rVert_{\Cal W}=
\int_R\lvert t\rvert\lVert T_t(f)\rVert_1dt=
2\pi\lVert T\rVert_{\Cal B_1^\oplus}$\,.
For a $C^1$-function $f\in A(\He)$ with $D(f)\in L^1(\He)$ we have
(using the extension of $\Cal F$ to
$A(\He)$) \;$\Cal F(D(f))=\widehat D(\Cal F(f))$\,.
$\Cal W\cap\Cal B_1^\oplus$ is dense in $\Cal W$\,. The main
observation of \cite{CG2} is now that dual convolution $\sharp$ can
be extended to make $\Cal W$ a (Banach) $\Cal B_1^\oplus$-module
(hence also an $A(\He)$-module). $D$ defines certainly a derivation on
the space of $C^1$-functions, thus by density and continuity
$\widehat D\circ\Cal F\!:A(\He)\to\Cal W$ becomes a derivation.
\\[0mm plus 1mm]
Products of representation coefficients are coefficients of the tensor product
of the representations
$(\pi_1(x)f_1\vert g_1)\,(\pi_2(x)f_2\vert g_2)=
((\pi_1\otimes\pi_2)(x)(f_1\otimes f_2)\vert\,g_1\otimes g_2)$\,.
Formulas for $\sharp$ are obtained by considering decompositions of
$\pi_1\otimes\pi_2$ into irreducible representations. For the Heisenberg
group $\He$ this was done in \cite{CG2}\,sec.\,5 (their notation is
slightly different, defining $\sharp$ as a product on
$L^1(\widehat\He,dt,\Cal S_1(\Cal H))$ which is isomorphic to
$\Cal B_1^\oplus(\widehat\He,\nu)$ by the mapping
$(\Cal T F)_t=\lvert t\rvert^{-1}F_t$ \;--- this simplifies some formulas).
More abstractly, $\pi_t$ can be obtained as an induced representation
from the subgroup $\R^2$ and then one can apply e.g. \cite{KL}\,Th.\,1.4.11
(a special case of Mackey's tensor product theorem). This gives
$\pi_{t_1}\otimes\pi_{t_2}\simeq\infty\,\pi_{t_1+t_2}$\,. But we need an
explicit description of the equivalence.
\\[1mm plus 2mm]
$\Cal H\otimes_2\Cal H$ denotes the (completed)
Hilbert space tensor product and we use the identification with
$L^2(\R^2)$ \,(be aware that \cite{F}\,p.\,9f.
uses a different notion of tensor product which is quite non-standard,
compare \cite{Fo}\,7.3).
For $t_1+t_2\neq0$, we consider a unimodular coordinate transformations giving rise to a unitary
operator
\;$(W_{t_1,t_2}f)(\psi_1,\psi_2)=
f(\psi_1-\frac{t_2}{t_1+t_2}\psi_2,\,
\psi_1+\frac{t_1}{t_1+t_2}\psi_2)$
for $f\in L^2(\R^2)\,$. Then
$W_{t_1,t_2}\circ((\pi_{t_1}\otimes\pi_{t_2})(h))=
(\pi_{t_1+t_2}(h)\otimes1)\circ W_{t_1,t_2}$ for $h\in\He$ \,($1$ denoting
the identity operator) and the mapping $(t_1,t_2)\mapsto W_{t_1,t_2}$ is
continuous for the strong operator topology.
Let $1\otimes\tr:\Cal S^1(\Cal H\otimes_2\Cal H)\to\Cal S^1(\Cal H)$ be the
dual operator for the embedding
$A\mapsto A\otimes 1\ \:(\,\Cal B(\Cal H)\to\Cal B(\Cal H\otimes_2\Cal H)\,)$,
\;i.e. $\tr(A\circ(1\otimes\tr)(B)\,)=\tr((A\otimes 1)\circ B)$ for
$A\in\Cal B(\Cal H),\;B\in\Cal B(\Cal H\otimes_2\Cal H)$. For
$A_1,A_2\in\Cal S^1(\Cal H)$ put
\;$\theta(A_1,A_2,t_1,t_2)=\linebreak
(1\otimes\tr)(W_{t_1,t_2}(A_1\otimes A_2)W_{t_1,t_2}^*)\
(\in\Cal S^1(\Cal H)\,)$\,.
This describes the products of coefficients of $\pi_{t_1},\pi_{t_2}\,$,
i.e. $\tr(A_1\pi_{t_1}(h)^*)\mspace{5mu}\tr(A_2\pi_{t_2}(h)^*)=
\tr(\,\theta(A_1,A_2,t_1,t_2)\:\pi_{t_1+t_2}(h)^*\,)$ for\linebreak
$h\in\He$\,. $\theta$ is continuous in all variables (for $\lVert\ \rVert_1$
and $t_1+t_2\neq0$).
It follows that (in our notation) \,$\lvert t\rvert\,(S\sharp T)_t=
\frac1{2\pi}\int_\R
\theta(\lvert s\rvert S_s,\lvert t-s\rvert T_{t-s},s,t-s)\,ds$ for
$S,T\in\Cal B_1^\oplus(\widehat\He,\nu)$\,.
\\[4mm plus 2mm]
Recall that $\Cal F(uv)=\Cal F(u)\,\sharp\,\Cal F(v)$ for $u,v\in A(\He)$ and
that $\Cal F$ extends to an isometric isomorphism
$\VN(\He)\to L^\infty(\R,\nu,\Cal B(\Cal H))$\,. This provides estimates
for the operator norm
$\lVert S\sharp T\rVert\le\lVert S\rVert_1\lVert T\rVert$ for
$S\in\Cal B_1^\oplus,\ T\in\Cal B_1^\oplus\cap L^\infty\
(=\Cal F(A(\He)\cap\VN(\He))\,)$\,. Considering this for
$S_t=f(t)A_1\,,\:T_t=g(t)A_2$ with continuous $f,g$ it follows by standard
arguments that
$\lVert\theta(A_1,A_2,t_1,t_2)\rVert\le
\frac{\lvert t_2\rvert}{\lvert t_1+t_2\rvert}
\lVert A_1\rVert_1\lVert\, A_2\rVert$ for
$A_1,A_2\in\Cal S^1(\Cal H),\;t_1+t_2\neq0$ \,(it does not look easy to prove
this more directly, alternatively one can also use arguments as in the proof
of \cite{CG2}\,Th.\,4.3 on p.\,2458). This easily implies
$\lVert S\sharp T\rVert_{\Cal W}\le\lVert S\rVert_1\lVert T\rVert_{\Cal W}$
for $S\in\Cal B_1^\oplus,\ T\in\Cal W\cap\Cal B_1^\oplus$ and this allows to
extend~$\sharp$\,, making $\Cal W$ a $\Cal B_1^\oplus$-module.
%
\section{The groups $\He_{\Cal K}$}\vspace{1mm} \label{heisK}
The reduced Heisenberg group has been treated in \cite{CG}\,sec.\,6 and
for general groups $\He_{\Cal K}$ the method works similarly. We take
up the notations of Section\,\ref{mini},
$\He_{\Cal K}=(\R\times\Cal K)\rtimes\R$\,, $\Cal K$ is a non-trivial
compact group and
we have given a continuous homomorphism
$\varphi_0\!\!:\R\to \Cal K$ with dense image, $\Cal K$ is written additively.
The action of $\R$ on $\R\times\Cal K$ is defined by
$x\circ(y,k)=(y,k+\varphi_0(xy))\quad(x,y\in \R\,,\,k\in\Cal K$\,).
Then $\varphi(y,z)=(y,\varphi_0(z))$  extends to a continuous homomorphism
$\varphi\!:\He\to\He_{\Cal K}$ with dense image.
$j_{\Cal K}(k)=(0,k)\ (j_{\Cal K}\!:\Cal K\to\He_{\Cal K})$ maps $\Cal K$
isomorphically onto the centre of $\He_{\Cal K}$\,, similarly with
$j_Z(z)=(0,z)\ (j_Z\!:\R\to\He)$\,,  and we have
$\varphi\circ j_Z=j_{\Cal K}\circ\varphi_0$\,.
For Haar
measure we take the product of standard Lebesgue measure on $\R^2$
and normalized Haar measure on $\Cal K$\,. $\He_{\Cal K}$ is unimodular.
The dual group
$\Gamma=\widehat{\Cal K}$ is considered as a subgroup of $\R$ in such a
way that for $t\in\Gamma$ the corresponding character $\chi_t$ of $\Cal K$
satisfies $\chi_t(\varphi_0(z))=e^{itz}$ for $z\in\R$ \,(see Rem.\,(b) in
Section\,\ref{mini} and Rem.\,(c) in Appendix\,\ref{dense}).
\\[0mm plus 1mm]
If $\pi$ is an irreducible representation of $\He_{\Cal K}$ then
$\pi\circ\varphi$ must be irreducible. Hence if $\pi$ is infinite
dimensional it follows that $\pi\circ\varphi\simeq\pi_t$ for some $t\neq0$\,.
Conversely, looking at the definition of $\pi_t$ in Sec.\,\ref{heis},
it is easy to see that $\pi_t$ can be written as $\pi\circ\varphi$ iff
the mapping \,$\varphi_0(z)\mapsto e^{itz}$ is well defined on $\varphi_0(\R)$
and extends continuously to $\Cal K$ \,(then defining a character of
$\Cal K$). By our identification this is equivalent to
$t\in\Gamma\setminus\{0\}$\,. The corresponding representation induced
on $\He_{\Cal K}$ will be denoted again by $\pi_t$\,. The finite dimensional
irreducible representations of $\He$ are trivial on $Z$ and one dimensional.
Hence they
are given by characters on $\He/Z\cong\R^2$ (projecting $(y,z)x\in\He$ to
$(x,y)\in\R^2$) and every character of $\R^2$ defines also a one dimensional
representation of $\He_{\Cal K}$\,. As above we identify $(\R^2)\sphat$ \;with
$\R^2$ associating to $w\in\R^2$ the character
$\chi_w(v)=e^{i(v\vert w)}\ (v\in\R^2)$. The corresponding representation of
$\He_{\Cal K}$ (and its extension to $L^1$ and distributions) will also
be denoted by $\chi_w$\,. With these notations and parametrizations
we get $(\He_{\Cal K})\sphat=(\Gamma\setminus\{0\})\cup\R^2$.
It is not hard to see that both parts are open, on $\R^2$ the dual induces
the euclidian topology and on  $\Gamma\setminus\{0\}$ the discrete topology
(in particular $(\He_{\Cal K})\sphat$ \;is Hausdorff; $\He_{\Cal K}$ are
$FC^-$-groups, general results for the dual of $FC^-$-groups have been shown
in \cite{Li}).\vspace{0mm plus 1mm}
\begin{Lem}   \label{plheisK}
Using Haar measure and parametrization as above, the Plancherel measure
$\nu$ of $\He_{\Cal K}$ satisfies
$d\nu=\frac1{4\pi^2}dw$ on $\R^2$ and $\nu(\{t\})=\frac{\lvert t\rvert}{2\pi}$
for $t\in\Gamma\setminus\{0\}$.
\end{Lem}
\begin{proof}
This follows by similar computations as for the Heisenberg group
(\cite{Fo}\,7.6.1).
\end{proof}\noindent
If $\Cal K$ is not metrizable (equivalently, if $\Gamma$ is uncountable)
this gives an example of the Plancherel
theorem for a group that is non-compact, $\sigma$-compact, but not second
countable -- of course
the Plancherel measure is not $\sigma$-finite in that case (see also
Rem.\,(d) in Appendix\,\ref{AppenB}). It follows that
$\Cal B_1^\oplus((\He_{\Cal K})\sphat\,,\nu)$ is the sum of (complex valued)
$L^1(\R^2,\frac1{4\pi^2}dw)$ and a ($\Cal S^1(\Cal H)$-valued) weighted
$l^1$-space.
$\lVert T\rVert_1=\frac1{4\pi^2}\int_{\R^2}\,\lvert T_w\rvert\,dw+
\frac1{2\pi}\sum_{t\in\Gamma\setminus\{0\}}
\lvert t\rvert\,\lVert T_t\rVert_1$\,.
\\
We project now the distribution from Section\,\ref{heis} to $\He_{\Cal K}$\,,
putting
$d=\lim_{z\to0}\linebreak
\frac1z(\delta_{\varphi(0,-z)}-\delta_{\varphi(0,0)})$
and $D(f)=f\star d$\,. Recall that $\varphi(0,z)=\varphi_0(z)$. If $\Cal K$
is a Lie group (equivalently, if $\Gamma$ is finitely generated), $\varphi_0$
is analytic (\cite{V}\,L.\,2.11.1 or more directly, using a duality argument)
and it follows that $d$ and $D$ are defined for\linebreak
$f\in C^1(\He_{\Cal K})$
($\He_{\Cal K}$ being again a Lie group). In the general case, $\Cal K$ is
the projective limit of Lie groups $\Cal K/K_1$ (e.g. taking all finitely
generated subgroups of $\Gamma$) and
$\He_{\Cal K/K_1}\cong\He_{\Cal K}/j_{\Cal K}(K_1)$.
With the usual embeddings (see also Rem.\,(d) in Appendix\,\ref{AppenB})
\,$d$ and $D$ are defined for
$f\in\bigcup\,\{C^1(\He_{\Cal K/K_1}):K_1\text{ a closed subgroup of }\Cal K
\text{ with}\linebreak\Cal K/K_1\text{ Lie\,}\}$. It is easy to see that
$\chi_w(d)=0$ for $w\in\R^2$ and (as in Section\,\ref{heis})
\,$(\pi_t(d)g)(\psi)=-itg(\psi)$
for $g\in\Cal H$ and $t\in\Gamma\setminus\{0\}$. Then
(similar to Section\,\ref{comp}) our formula for the Plancherel measure
(Lemma\,\ref{plheisK}) implies that
$\lVert f\star d\rVert_{\VN}\le2\pi\lVert f\rVert_A$ for
$f\in
A(\He_{\Cal K})\cap L^1(\He_{\Cal K})\cap\bigcup C^1(\He_{\Cal K/K_1})$ \,
with $f\star d\in L^1$ (e.g. $f$ of compact support).
It follows that $D$ extends to a derivation $A(G)\to\VN(G)$\,.
\begin{Rems}
(a) We recall from \cite{CG}\,2.1, for a Banach algebra $A$\,, a
derivation\linebreak $D\!:A\to A'$ is called cyclic if $(Da,b)=-(Db,a)$
holds for
$a,b\in A$ \;(\,$(\;,\,)$ denoting the evaluation of functionals).
Gr\o nb\ae k
called $A$ cyclically amenable if there are no non-zero bounded cyclic
derivations. For $\SO(3),\,\SU(2),\,ax+b,\,\He_{\Cal K}$ the derivations
that we considered are cyclic. By \cite{CG}\,Prop.\,2.5, for any locally
compact group $G$ that has one of these groups as a closed subgroup the
Fourier algebra $A(G)$ is not cyclically amenable. But the situation is
not clear for the groups $\He,\,\widetilde{M(2)},\,M(2)_{\Cal K},\,\G_a$\,.
\item[(b)] The distribution $d$ that we used for $\He_{\Cal K}$ is central,
thus $D^{(r)}=D\ \;(D$ is left and right invariant). For derivatives
involving a non-central direction one
does not get a bounded mapping $A(\He_{\Cal K})\to\VN(\He_{\Cal K})$\,.
Hence the present method produces just a one dimensional space of
derivations $A(\He_{\Cal K})\to\VN(\He_{\Cal K})$.
\end{Rems}
\vspace{.5mm plus 2mm}
\section{The Euclidian motion group $M(2)$}\vspace{1mm} \label{euclm}
$M(2)$ is not covered by \cite{CG},\cite{CG2}. In \cite{LLSS}\,sec.\,2.2
it was shown that\linebreak $A(M(2))$ is not weakly amenable, using
spectral synthesis.
We develop here our method in detail, since the constructions are basic
for the treatment of $M(2)_{\Cal K}$ in the next section. Defining
derivations as before, it turns out that they do not give bounded mappings
$A(M(2))\to\VN(M(2))$. Working on the dual side, we are constructing a
$\Cal B_1^\oplus$-module $\Cal W$ such that
$\widehat D\!:\Cal B_1^\oplus\to\Cal W$ is bounded. First we will derive a
formula for the dual convolution $\sharp$\,. For $M(2)$ it turns out that 
$\sharp$ behaves quite differently compared to the case of $\He$\,. We did
not succeed to find estimates making it possible to define $\Cal W$ similarly
as in Section\,\ref{heis} as a weighted $L^1$-space of $\Cal K(\Cal H)$-valued
(or $\Cal S^1(\Cal H)$-valued) functions. Recall that if $\Cal H=L^2(\Omega)$,
trace class operators are integral operators and for a finite measure
space this gives embeddings \,$\Cal S^1(\Cal H)\subseteq\Cal S^2(\Cal H)=
L^2(\Omega^2)\subseteq L^1(\Omega^2)$. In our case, for the infinite
dimensional representations we may choose $\Omega=\T=\R/(2\pi\Z)$ with
normalized Haar measure $\frac1{2\pi}d\psi$\,. We will define $\Cal W$
as a subspace of the weighted $L^1$-space
$\overline{\Cal W}= L^1(]0,\infty[,w(\rho)\,d\rho,L^1(\T^2))$ with
the weight function $w(\rho)=\dfrac\rho{1+\rho^2}$\,, leading (partially)
outside $\Cal B(\Cal H)$.
\\[0mm plus 1mm]
Similar to Section\,\ref{heis} we use the parametrization
$M(2)=\R^2\rtimes\T$\,. With the rotation matrices
$A_\alpha=\begin{pmatrix}\cos\alpha &-\sin\alpha\\
\sin\alpha & \phantom{-}\cos\alpha
\end{pmatrix}=\exp\bigl(\,\alpha\begin{pmatrix}0 &-1\\1 & 0 \end{pmatrix}\bigr)$
($\alpha\in\R$, but depending only on the coset$\mod2\pi$), the action is
defined by $\alpha\circ v=A_\alpha v$ ($v\in\R^2$). For Haar
measure we take $\dfrac1{(2\pi)^2}\,dv\,d\alpha$
\,(\cite{Su}\,Ch.\,IV,\,(3.1)). $M(2)$ is unimodular.
\\[1mm plus 1mm]
The representation
theory is again classical. Take $\Cal H=L^2(\R,\frac1{2\pi}d\psi)$.
For real $\rho>0$ an
irreducible representation $\pi_\rho$ is defined for $f\in\Cal H$ by
$(\pi_\rho(v)f)(\psi)=e^{i\rho\,r\cos(\phi-\psi)}f(\psi)\,,\
(\pi_\rho(\alpha)f)(\psi)=f(\psi-\alpha)\
\;(\alpha,\psi\in\T\,,\;v=r\,\binom{\cos\phi}{\sin\phi}\in\R^2$,
polar coordinates)\,. $M(2)$ is of type\;I\,. The representations
$\pi_\rho$ are
pairwise non-equivalent and exhaust (up to equivalence) all the infinite
dimensional irreducible representations of $M(2)$\,. For usage in the
next section, we discuss also the finite dimensional representations.
We have $M(2)/\R^2\cong\T$\,. For $n\in\Z\,,\ \chi_n(\psi)=e^{in\psi}\
(\psi\in\T)$ lifts to a one dimensional representation of $M(2)$ (again
denoted by $\chi_n$).
We have $\widehat{M(2)}=\{\chi_n:n\in\Z\}\cup\{\pi_\rho:\rho>0\}$ (more
precisely: this gives a complete set of representatives).
$\{\pi_\rho:\rho>0\}$ is an open subset and $\widehat{M(2)}$ induces the
standard topology of $]0,\infty[$\,. On $\{\chi_n:n\in\Z\}$ we get the
discrete topology, $\overline{\{\pi_\rho\}}=\{\pi_\rho\}\cup\{\chi_n:n\in\Z\}$
for each $\rho>0$\,.
The Plancherel measure is concentrated on $\{\pi_\rho:\rho>0\}$ and with the
given parametrization of the representations we have
$d\nu=\rho\,d\rho$ \,(\cite{Su}\,IV,\,Th.\,4.2, \cite{Vi}\,Ch.\,IV\,\S5 uses a
different normalization of Haar measure, resulting in a compensating factor in
the Plancherel measure). Hence
$\Cal B_1^\oplus(\widehat{M(2)},\nu)\cong
L^1(]0,\infty[,\rho\,d\rho,\Cal S_1(\Cal H))$, general elements of
$\Cal B_1^\oplus(\widehat{M(2)},\nu)$ will be denoted as
\,$T=(T_\rho)_{\rho>0}$\,.
\\[2mm]
For the direction $e_1=\binom10\in\R^2$ we will consider,
similarly as before, the distribution
$d=\lim_{y\to0}\frac1y(\delta_{-ye_1}-\delta_{0e_1})$ and
$D(f)=f\star d$\,.
It follows that $(\pi_\rho(d)g)(\psi)=-i\rho\cos(\psi)\,g(\psi)$
\,(for $g\in\Cal H$) and
it is easy to see that $D$ does {\it not} produce a bounded mapping
$A(M(2))\to\VN(M(2))$\,. As sketched above, we will next investigate the
dual convolution.
\\
For the decomposition of tensor products, we follow \cite{Vi}\,Ch.\,IV,\,\S6.
\,$\Cal H\otimes_2\Cal H$ (completed
Hilbert space tensor product) is identified with
$L^2(\T^2,\frac1{(2\pi)^2}d(\psi_1,\psi_2)$. For $f\in\Cal H\otimes_2\Cal H$
define \,$(Wf)(\mu)(\psi)=f(\psi,\psi+\mu)$. It follows from standard
measure theory that $(Wf)(\mu)\in\Cal H$ for a.e.\,$\mu$ and
$W\!:\Cal H\otimes\Cal H\to L^2(\T,\frac1{2\pi}d\mu,\Cal H)$ gives
an isometric isomorphism.
\\
For $\varphi\in\T\,,\ g\in\Cal H$\,, we define
$(U_\varphi g)(\psi)=g(\psi-\varphi)$, giving a unitary operator on
$\Cal H$\,, depending strongly continuously on $\varphi$\,.  Fix now
$\rho_1,\rho_2>0$ and for $\mu\in\R$ put
$\rho_1+\rho_2\,e^{i\mu}=\rho\,e^{i\varphi}$, defining functions
$\rho=\rho(\rho_1,\rho_2,\mu)=\lvert\rho_1+\rho_2\,e^{i\mu}\rvert$ and
$\varphi=\varphi(\rho_1,\rho_2,\mu)$ \,(using only the cosets of
$\mu,\varphi\mod2\pi$)\,. Finally, we define
\,$(W_{\!\rho_1,\rho_2}f)(\mu)=U_\varphi((Wf)(\mu))$
\,($f\in\Cal H\otimes_2\Cal H$), giving again an isometric isomorphism
$W_{\!\rho_1,\rho_2}\!:\Cal H\otimes_2\Cal H\to
L^2(\T,\frac1{2\pi}d\mu,\Cal H)$,
\;$(\rho_1,\rho_2)\mapsto W_{\!\rho_1,\rho_2}$ is strongly continuous.
For $n\in\Z$ we define $M_ng=\chi_n\cdot g$ giving also a unitary operator on
$\Cal H$\,.

\begin{Pro} \label{tenseu}  
For $\rho_1,\rho_2>0,\ n\in\Z,\ x\in M(2)$\,, we have\vspace{-1.5mm}
\begin{gather*}
W_{\!\rho_1,\rho_2}\circ(\pi_{\rho_1}\otimes\pi_{\rho_2})(x)
\circ W_{\!\rho_1,\rho_2}^*=\int_\T^\oplus\pi_\rho(x)\,d\mu
\\[1.5mm]
(\chi_n\otimes\pi_\rho)(x)\;=\;M_n\pi_\rho(x)M_n^*\,.
\end{gather*}
\end{Pro}\noindent
The first formula gives a decomposable operator on
$L^2(\T,\frac1{2\pi}d\mu,\Cal H)$
\;(\cite{T1}\,Ch.\,IV, Def.\,7.9). Both generalize to the extensions of
the representations to $L^1(M(2))$.
\begin{proof}
Take $f\in\Cal H\otimes_2\Cal H$ and put $g=W_{\!\rho_1,\rho_2}f$\,.
Then $f=W_{\!\rho_1,\rho_2}^*g$ and $g(\mu)(\psi)=(Wf)(\mu)(\psi-\varphi)=
f(\psi-\varphi,\psi-\varphi+\mu)$.
It is enough to consider the case $f=f_1\otimes f_2$\,, \,then
$f(\psi-\varphi,\psi-\varphi+\mu)=f_1(\psi-\varphi)\,f_2(\psi-\varphi+\mu)$
and $(\pi_{\rho_1}\otimes\pi_{\rho_2})(x)f=
(\pi_{\rho_1}(x)f_1)\otimes(\pi_{\rho_2}(x)f_2)$.
For $\alpha\in\T$ an easy computation gives
$[(W_{\!\rho_1,\rho_2}\circ(\pi_{\rho_1}\otimes\pi_{\rho_2})(\alpha))f](\mu)
(\psi)=$\newline$f_1(\psi-\varphi-\alpha)\,
f_2(\psi+\mu-\varphi-\alpha)=[\pi_\rho(\alpha)(g(\mu))](\psi)$\,,
and vor $v=r\,\binom{\cos\phi}{\sin\phi}\in\R^2$, using
$\rho_1\binom{\cos\psi}{\sin\psi}+\rho_2\binom{\cos(\psi+\mu)}{\sin(\psi+\mu)}
=\rho\binom{\cos(\psi+\varphi)}{\sin(\psi+\varphi)}$ \,one gets \
\vspace{.5mm}
$[(W_{\!\rho_1,\rho_2}\circ(\pi_{\rho_1}\otimes\pi_{\rho_2})(v))f](\mu)
(\psi)=$\newline$e^{i\rho r\cos(\phi-\psi)}f_1(\psi-\varphi)\,
f_2(\psi+\mu-\varphi)=[\pi_\rho(v)(g(\mu))](\psi)$ showing the first formula.
\\[2mm]
We identify $\C\otimes\Cal H$ with $\Cal H$ (via
$\gamma\otimes h\mapsto\gamma h$) and
consider $\chi_n\otimes\pi_\rho$ as a representation on $\Cal H$\,. The
second formula is checked easily  (it implements the equivalence\linebreak
$\chi_n\otimes\pi_\rho\simeq\pi_\rho$).
\end{proof}\noindent
The embedding of the subalgebra $L^\infty(\T)\overline\otimes\Cal B(\Cal H)$
of decomposable operators into the algebra $\Cal B(L^2(\T,\Cal H))$ of all
bounded operators defines a projection
\,$\pr_{\Cal S}\!:{ } \linebreak\Cal S^1(L^2(\T,\Cal H))
\to L^1(\T,\Cal S^1(\Cal H))$ of
the preduals (using \cite{T1}\,Ch.\,IV,\,Th.\,7.17).
For $A_1,A_2\in\Cal S^1(\Cal H)$ define
\[
\theta_0(A_1,A_2)=\pr_{\Cal S}(W_{\!\rho_1,\rho_2}\circ(A_1\otimes A_2)
\circ W_{\!\rho_1,\rho_2}^*) \quad (\in L^1(\T,\frac1{2\pi}d\mu,
\Cal S^1(\Cal H))
\]
(a more explicit description in terms of
integration kernels will be given in Lemma\,\ref{integk} below).
\begin{Cor1}	\label{prmu}  
For $A_1,A_2\in\Cal S^1(\Cal H)$ consider coefficients
\,$u_j(x)=\tr(A_j\,\pi_{\rho_j}(x)^*)\linebreak (j=1,2)$. Then
\[ u_1(x)\,u_2(x)=\frac1{2\pi}\,\int_0^{2\pi}
\tr\bigl(\,\pi_\rho(x)^*\,\theta_0(A_1,A_2)(\mu)\,\bigr)\,d\mu\,.
\]
\end{Cor1}
\begin{proof}
It is enough to show this for $A_jh=(h\vert h_j)k_j$\,, where
$h_j,k_j\in\Cal H\ \ (j=1,2)$. Then\quad
$u_j(x)=(k_j\vert\,\pi_{\rho_j}(x)h_j),\quad
u_1(x)\,u_2(x)=
(\,k_1\otimes k_2\,\vert
\;[(\pi_{\rho_1}\otimes\pi_{\rho_2})(x)](h_1\otimes h_2)\,)$
\;and the formula follows by a short computation from
Proposition\,\ref{tenseu}.
\end{proof}\noindent
By a change of coordinates, this can be expressed in terms of the
Plancherel measure. We have
\,$\cos\mu=\dfrac{\rho^2-\rho_1^2-\rho_2^2}{2\rho_1\rho_2}$\,, \vspace{1mm}
hence
$\mu\mapsto\rho(\rho_1,\rho_2,\mu)$ (briefly $\rho(\mu)$\,) is injective on
$[0,\pi]$ and on $[\pi,2\pi]$ and it maps both intervals onto
$[\lvert\rho_1-\rho_2\rvert,\rho_1+\rho_2]$\,.
In addition $\rho(\mu)=\rho(2\pi-\mu)$.
For $A_1,A_2\in\Cal S^1(\Cal H),\ \rho=\rho(\mu)$, we define
\[
\theta(A_1,A_2)(\rho)=\theta_0(A_1,A_2)(\mu)+\theta_0(A_1,A_2)(2\pi-\mu)
\]
($\theta_0,\theta$ depend on $\rho_1,\rho_2$ as well, sometimes we will
write $\theta(A_1,A_2,\rho_1,\rho_2)$\,). For
$\lvert\rho_1-\rho_2\rvert<\rho<\rho_1+\rho_2$ \,put
\,$K(\rho_1,\rho_2,\rho)=
\dfrac2{\pi\sqrt{4\rho_1^2\rho_2^2-(\rho^2-\rho_1^2-\rho_2^2)^2}}$ \
\ \vspace{1.5mm}(this is the kernel defining the Bessel-Kingman hypergroup --
\cite{BH}\,p.\,235 with $\alpha=0$).
\begin{Cor1}	 \label{prrh}	   
For $A_1,A_2,u_1,u_2$ as in Corollary\,\ref{prmu}, we have \
$\theta(A_1,A_2,\rho_1,\rho_2)\in
L^1(\,[\lvert\rho_1-\rho_2\rvert,\rho_1+\rho_2]\,,
\frac12K(\rho_1,\rho_2,\rho)\,\rho\,d\rho\,,\Cal S^1(\Cal H)\,)$ \ and
\[ u_1(x)\,u_2(x)=
\,\int_{\lvert\rho_1-\rho_2\rvert}^{\rho_1+\rho_2}
\frac12\,\tr\bigl(\,\pi_\rho(x)^*\,\theta(A_1,A_2)(\rho)\,\bigr)
\,K(\rho_1,\rho_2,\rho)\,d\nu(\rho)\,.
\]
\end{Cor1}\noindent
Thus $u_1\cdot u_2=
\Cal F^{-1}\bigl(\,(\frac12\,\theta(A_1,A_2)(\rho)\,K(\rho_1,\rho_2,\rho))_{
\lvert\rho_1-\rho_2\rvert<\rho<\rho_1+\rho_2}\bigr)\in A(M(2))$
(the operator field is
meant to be extended by $0$ for $\rho$ outside the interval). In particular,
one gets that \;$\pi_{\rho_1}\otimes\pi_{\rho_2}\simeq
\,2\,\int_{\lvert\rho_1-\rho_2\rvert}^{\rho_1+\rho_2}\pi_{\rho}\,d\rho$\,,
but the formula above gives more detailed information for the product
of coefficients.
\begin{proof}
Again, an elementary computation shows that under the substitution,\linebreak
$\frac1{2\pi}d\mu=\frac12K(\rho_1,\rho_2,\rho)\,\rho\,d\rho$ holds for
$\mu\in[0,\pi]$ and similarly for $\mu\in[\pi,2\pi]$.
\end{proof}
\begin{Cor1}	 \label{dualc}	   
For
\;$S=(S_\rho),T=(T_\rho)\in \Cal B_1^\oplus(\widehat{M(2)},\nu)$, we get
\[
(S\sharp T)_\rho=\iint_{D_\rho}\frac12\,
\theta(S_{\rho_1},T_{\rho_2})(\rho)\,K(\rho_1,\rho_2,\rho)\,
\rho_1\rho_2\,d(\rho_1,\rho_2)
\]
where $D_\rho$ denotes the half-strip $\{(\rho_1,\rho_2):\;\rho_1,\rho_2>0,\,
\lvert\rho_1-\rho_2\rvert<\rho<\rho_1+\rho_2\}$.
\end{Cor1}\noindent
It follows that if $S_\rho=0$ for $\rho\notin[a_1,b_1]$,
$T_\rho=0$ for $\rho\notin[a_2,b_2]$ then
$(S\sharp T)_\rho=0$ for
$\rho\notin\,]\mspace{1mu}\lvert a_1-a_2\rvert\,,b_1+b_2[$ \,(see also
Lemma\,\ref{convo}).
One can give similar expressions for
products of functions from the Fourier-Stieltjes algebra, using arbitrary
finite measures on $]0,\infty[$ and adding contributions from the
one dimensional representations.
\begin{proof}
For fixed $\rho_1,\rho_2>0\,,\ A_1,A_2\in\Cal S^1(\Cal H)$, we have noted
after Corollary\,\ref{prrh} that $\rho\mapsto
\frac12\,\theta(A_1,A_2)(\rho)\,K(\rho_1,\rho_2,\rho)$ (extended by $0$
outside $]\lvert\rho_1-\rho_2\rvert,\rho_1+\rho_2[$\,) defines an element
$\dc(A_1,A_2,\rho_1,\rho_2)$ of $\Cal B_1^\oplus(\widehat{M(2)})$ of norm
$\le\lVert A_1\rVert_1\lVert A_2\rVert_1$\,. Once we can show that
for $S,T\in\Cal B_1^\oplus(\widehat{M(2)})\ \
(\rho_1,\rho_2)\mapsto\dc(S_{\rho_1},T_{\rho_2},\rho_1,\rho_2)$ defines a
measurable mapping \,$]0,\infty[^2\to\Cal B_1^\oplus(\widehat{M(2)})$\,,
its integrability (with respect to $\nu\otimes\nu$) follows immediately,
hence the integral in the Corollary defines an element $\dc(S,T)$ of
$\Cal B_1^\oplus(\widehat{M(2)})$. Then, by Fubini's theorem and
Corollary\,\ref{prrh} it is not hard to see that
$\Cal F^{-1}(\dc(S,T))=\Cal F^{-1}(S)\,\Cal F^{-1}(T)$\,, hence
\,$\dc(S,T)=S\sharp T$ \,proving the Corollary.
\\[2mm]
Since $(\rho_1,\rho_2)\mapsto W_{\!\rho_1,\rho_2}$ is strongly continuous
it follows that
$(\rho_1,\rho_2)\mapsto W_{\!\rho_1,\rho_2}\circ\linebreak(A_1\otimes A_2)
\circ W_{\!\rho_1,\rho_2}^*$ is norm continuous for fixed
$A_1,A_2\in\Cal S^1(\Cal H)$. Consequently,
$(\rho_1,\rho_2)\mapsto\theta_0(A_1,A_2,\rho_1,\rho_2)\;
\in L^1(\T,\Cal S^1(\Cal H))$ is norm continuous. From basic
properties of the Bochner integral (see e.g.\,\cite{DS}\,Th.\,III.11.17) it
follows that
$(\rho_1,\rho_2,\mu)\mapsto\theta_0(A_1,A_2,\rho_1,\rho_2)(\mu)\;
\in\Cal S^1(\Cal H)$ is (a.e.) measurable on $]0,\infty[^2\times\T$\,.
The substitution $\rho=\rho(\mu)$ being continuous, we get that
$(\rho_1,\rho_2,\rho)\mapsto\theta(A_1,A_2,\rho_1,\rho_2)(\rho)$ is
measurable on
$\{(\rho_1,\rho_2,\rho):\ \lvert\rho_1-\rho_2\rvert<\rho<\rho_1+\rho_2\}$.
Then by standard measure theoretical approximation techniques it follows
that
$(\rho_1,\rho_2,\rho)\mapsto\theta(S_{\rho_1},T_{\rho_2},\rho_1,\rho_2)(\rho)$
is measurable for $S=(S_\rho),T=(T_\rho)\in \Cal B_1^\oplus(\widehat{M(2)})$.
$(\rho_1,\rho_2,\rho)\mapsto K(\rho_1,\rho_2,\rho)$ being continuous,
we get measurability of
$(\rho_1,\rho_2)\mapsto\dc(S_{\rho_1},T_{\rho_2},\rho_1,\rho_2)$.
\end{proof}\noindent
We represent now as integral operators $(Ah)(\psi)=
\frac1{2\pi}\int_0^{2\pi} k_A(\psi,\psi_1)\,h(\psi_1)\,d\psi_1 \linebreak
(\psi\in\T\,,\;
h\in L^2(\T)\,;\ k_A\in L^2(\T^2)$ \,being called the integration kernel
for $A$). By classical results, $A\mapsto k_A$
defines an isometric isomorphism
\,$\Cal S^2(\Cal H)\to L^2(\T^2,\linebreak
\frac1{4\pi^2}d(\psi,\psi_1))$. Recall that
$\Cal S^1(\Cal H)\subseteq\Cal S^2(\Cal H)$, hence this applies for
$A\in\Cal S^1(\Cal H)$. For one dimensional operators $Ah=(h\vert h_0)k_0\,,\
h_0,k_0\in\Cal H$\, (as in the proof of Corollary\,\ref{prmu} above), one has
$k_A(\psi,\psi_1)=k_0(\psi)\,\overline{h_0(\psi_1)}$\,. In general, $k_A$
can also be obtained by considering $A\in\Cal S^1(\Cal H)$ as an absolutely
convergent sum of one dimensional operators. The operators $M_n\ (n\in\Z)$
defined before Proposition\,\ref{tenseu} induce operators
${\bf M}_n\!:\Cal S_1(\Cal H)\to\Cal S_1(\Cal H)$ by
${\bf M}_n(A)=M_n\circ A\circ{M_n}^*$.
\begin{Lem}	 \label{integk}	   
For $\rho_1,\rho_2>0\,,\ A_1,A_2\in\Cal S^1(\Cal H)$ we have for almost all
$\mu\in\T$
$k_{\theta_0(A_1,A_2)(\mu)}(\psi,\psi_1)\,=\,
k_{A_1}(\psi-\varphi\,,\psi_1-\varphi)\:
k_{A_2}(\psi-\varphi+\mu\,,\psi_1-\varphi+\mu)$.\\
$k_{{\bf M}_n(A)}(\psi,\psi_1)\,=\,\chi_n(\psi-\psi_1)\,k_A(\psi,\psi_1)$ for
$A\in\Cal S^1(\Cal H),\,n\in\Z$\,.\vspace{-1mm}
\end{Lem}\noindent
$\varphi=\varphi(\rho_1,\rho_2,\mu)$ is defined before
Proposition\,\ref{tenseu}. Adding the kernels obtained above for $\mu$
and $2\pi-\mu$ \,(observe that $\varphi(2\pi-\mu)=-\varphi(\mu)$) gives a
corresponding representation for
$\theta(A_1,A_2)(\rho)$\,. For $A_1,A_2$ one dimensional,
$\theta(A_1,A_2)(\rho)$ will (in general) be a two dimensional operator.
It is not hard to construct examples of (e.g.\;one dimensional)
$A_1,A_2\in\Cal S^1(\Cal H)$ such that (using the formula above)
$k_{\theta_0(A_1,A_2)(\mu)}\notin L^2$ for some $\mu$\,. Then, in general,
$\theta(A_1,A_2)(\rho)\in\Cal S^1(\Cal H)$ holds just $\rho$-a.e.
and $\rho\mapsto\lVert\theta(A_1,A_2)(\rho)\rVert_1$ is unbounded.
\begin{proof}
This follows again from the definitions by direct computation, starting with
one dimensional $A_1,A_2$\,.
\end{proof}\noindent
We extend the operators ${\bf M}_n$ to operator fields, defining for
$T=(T_\rho):\ {\bf M}_n(T)=({\bf M}_n(T_\rho))\ \;(n\in\Z)$.
\begin{Cor}	 	   
$\theta_0(A_1,{\bf M}_n(A_2))={\bf M}_n(\theta_0(A_1,A_2))$ for
$A_1,A_2\in\Cal S^1(\Cal H)$.\\
$S\:\sharp\:{\bf M}_n(T)={\bf M}_n(S\sharp T)$ for
$S,T\in \Cal B_1^\oplus(\widehat{M(2)},\nu)$.
\end{Cor}\noindent
We use the embedding $A\mapsto k_A$ to consider $\Cal S^1$ as a subset of
$L^1(\T^2,\frac1{(2\pi)^2}d(\psi,\psi_1))$. For better distinction,
we denote the $L^1$-norm as $\lVert\ \rVert_{L^1}\,,\ \lVert\ \rVert_1$
shall refer (as before) to the $\Cal S^1$-norm. Using the formula of
Lemma\,\ref{integk} we define $\theta(a_1,a_2)(\rho)$ for arbitrary
measurable functions $a_1,a_2$ on $\T^2$
\;(thus $\theta(k_{A_1},k_{A_2})(\rho)=k_{\theta(A_1,A_2)(\rho)}$\,).
Also the operators ${\bf M}_n\ (n\in\Z)$ extend to measurable functions
on $\T^2$, giving isometrical isomorphisms of $L^1(\T^2)$. As above, they
extend further to $L^1(\T^2)$-valued vector fields.\vspace{-2mm} 
\begin{Lem}	 \label{ineq}	   
\thetag{a}\ $\lVert k_A\rVert_{L^1}\le\lVert A\rVert_1$
\item[(b)]
$\lVert\theta(k_{A_1},k_{A_2})(\rho)\rVert_{L^1}\le
2\,\lVert A_1\rVert_1\,\lVert A_2\rVert_1$
\item[(c)]
$\lVert\theta(a_1,a_2)(\rho)\rVert_{L^1}\le
2\,\lVert a_1\rVert_{L^\infty}\lVert a_2\rVert_{L^1}$
\item[(d)]
$\lVert\theta(a_1,a_2)(\rho)\rVert_{L^\infty}\le
2\,\lVert a_1\rVert_{L^\infty}\lVert a_2\rVert_{L^\infty}$\,.
\end{Lem}\noindent
Note also that $L^\infty(\T^2)\cap\Cal S^1$ is dense in $\Cal S^1(\Cal H)$
for the $\Cal S^1$-norm, \;$\Cal S^1$ is dense in $L^1(\T^2)$ for the
$L^1$-norm.
\begin{proof}
For \thetag{a}: $\lVert k_A\rVert_{L^1}\le\lVert k_A\rVert_{L^2}=
\lVert A\rVert_2\le\lVert A\rVert_1$\,. For \thetag{b} use Lemma\,\ref{integk}
and the Cauchy-Schwarz inequality. \thetag{c},\thetag{d} are easy too.
\end{proof}\noindent
For a weight function $w$ on $]0,\infty[$ (i.e.
$w\!:\,]0,\infty[\to[0,\infty[$
is a measurable function; it will be specified later,
see Proposition\,\ref{wspec} and the Remark\,(a) below),
let
\[
\bW= L^1\bigl(\,]0,\infty[\,,\,w(\rho)\,d\rho\,,\,
L^1(\T^2,\frac1{(2\pi)^2}d(\psi,\psi_1)\,)\,\bigr)
\]
be the
$L^1(\T^2)$-valued {\it weighted} $L^1$-space on $]0,\infty[$. It consists of
(equivalence classes of) Bochner-measurable functions $t=(t_\rho)$ with
$t_\rho\in L^1(\T^2)$ and
$\lVert t\rVert_{\bW}=
\int_0^\infty\lVert t_\rho\rVert_{L^1}\,w(\rho)\,d\rho<\infty$\,.
\\
$\Cal W_0$ shall consist of
the (classes of) bounded, measurable, $L^\infty$-valued and
boundedly supported functions (i.e., $t_\rho\in L^\infty(\T^2)$,
$\sup_\rho\,\lVert t_\rho\rVert_{L^\infty}<\infty$ and there exists $b>0$
such that $t_\rho=0$ for $\rho>b$\,; in analogy to
\cite{T1}\,Ch.IV,\,Def.\,7.7 we do not require that the range of $t$ is
a.e.\,separable for the norm topology of $L^\infty$).
\\
$\bbW$ shall consist of the (classes of) Bochner-measurable functions
$t=(t_\rho)$ with $t_\rho\in L^1(\T^2)$ and
$\int_0^b\lVert t_\rho\rVert_{L^1}\,\rho\,d\rho<\infty$ \, for each $b>0$\,.
\begin{Lem}	 \label{wsub}	   
We have \,$\Cal W_0\,,\:\Cal B_1^\oplus(\widehat{M(2)},\nu)\subseteq\bbW$\,.
\\
$\Cal W_0\cap\Cal B_1^\oplus$ is dense in $\Cal B_1^\oplus$ (for the
$\Cal B_1^\oplus$-norm).
\\
If $w(\rho)\le\rho$ for all $\rho>0$\,, then
\,$\Cal W_0\,,\:\Cal B_1^\oplus\subseteq\bW$\,, both are dense (for the
$\bW$-norm).
\\
If $\frac{\rho}{w(\rho)}$ is bounded on $]0,b]$ for each $b>0$\,,
then $\bW\subseteq\bbW$\,.
\end{Lem}
\begin{proof}
Considering $\Cal S^1$ as a subset of
$L^1(\T^2,\frac1{(2\pi)^2}d(\psi,\psi_1))$ defines (using Lemma \ref{ineq}(a))
a contractive embedding of $\Cal B_1^\oplus(\widehat{M(2)},\nu)$ into
$L^1(]0,\infty[,\nu,L^1(\T^2))$ and we will identify now $\Cal B_1^\oplus$
with this subspace of $L^1$-valued functions.
The other properties
use standard estimates and approximation techniques (see also the comment
after Lemma\,\ref{ineq}).
\end{proof}
\begin{Lem}	\label{convo}	
Let $f,g\!:\,]0,\infty[\to[0,\infty[$ be measurable functions such that
$f(\rho)=0$ for $\rho>b$. Then for all $\rho>0$
\newline
$\iint_{D_\rho}f(\rho_1)\,g(\rho_2)\,K(\rho_1,\rho_2,\rho)\,
\rho_1\rho_2\,d(\rho_1,\rho_2)\,\le\,\sup_{\rho_1}f(\rho_1)\
\int_{\max(\rho-b,0)}^{\rho+b}g(\rho_2)\rho_2\,d\rho_2$\,.
\end{Lem}
\begin{proof}
See Proposition\,\ref{tenseu} for the definition of $D_\rho$\,.
This is a special case (called "straightforward" in \cite{BH}\,p.\,27) of
convolvability in the Bessel-Kingman hypergroup. More explicitely, one uses
that (for fixed $\rho,\rho_2$)
\,$\rho_1\mapsto K(\rho_1,\rho_2,\rho)\,\rho_1$ is a probability
distribution concentrated on $]\,\lvert\rho-\rho_2\rvert\,,\rho+\rho_2[$
and if $f$ vanishes outside $]0,b]$\,, then the domain of
integration in the left integral can be restricted to
\,$\max(\rho-b,0)<\rho_2<\rho+b$\,.
\end{proof}\noindent
Lemma\,\ref{convo} and Lemma\,\ref{ineq}(c) imply that if
$s\in\Cal W_0\,,\ t\in\bbW\,$, then
$(s\sharp t)_\rho$ can be defined for each $\rho>0$
by an integral as in Cor.\,\ref{dualc} to Prop.\,\ref{tenseu}, giving
a measurable function  $s\sharp t\!:\,]0,\infty[\to L^1(\T^2)$\,. This
extends the dual convolution of $\Cal B_1^\oplus$.
\begin{Lem}	\label{dcext}	   
\thetag{a}\ For $s\in \Cal W_0\,,\ t\in\bbW$\,, we have $s\sharp t\in\bbW$\,.
If $s_\rho=0$ for $\rho>b$ we have
$\lVert(s\sharp t)_\rho\rVert_{L^1}\le
2\sup_{\rho_1}\lVert s_{\rho_1}\rVert_{L^\infty}\
\int_0^{\rho+b}\lVert t_{\rho_2}\rVert_{L^1}\,\rho_2\,d\rho_2$ for each
$\rho>0$\,.
\item[(b)]
For $s,t\in\Cal W_0$\,, we have $s\sharp t\in\Cal W_0$\,.
\item[(c)]
For $s^{(1)},s^{(2)}\in\Cal W_0\cap\Cal B_1^\oplus,\ t\in\bbW$\,, we have
$(s^{(1)}\sharp s^{(2)})\,\sharp\,t=s^{(1)}\sharp\,(s^{(2)}\sharp\,t)$\,.
\item[(d)] For $n\in\Z\,,t\in\bbW$ we have ${\bf M}_n(t)\in\bbW$\,, \
$\Cal W_0,\bW$ are ${\bf M}_n$-invariant subspaces,
$\lVert{\bf M}_n(t)\rVert_{\bW}=\lVert t\rVert_{\bW}$\,, \ for
$s\in\Cal W_0\ \ s\,\sharp\,{\bf M}_n(t)={\bf M}_n(s\sharp t)$\,.
\end{Lem}
\begin{proof}
(a) follows from Lemma\,\ref{convo} and Lemma\,\ref{ineq}(c). Similarly,
Lemma\,\ref{ineq}(d) is used for (b). We know that (c) holds for
$t\in\Cal B_1^\oplus$. By (b) there exists $b>0$ such that
$s^{(1)}_\rho\,,\,s^{(2)}_\rho\,,\,(s^{(1)}\sharp s^{(2)})_\rho\,=0$ for
$\rho>b$.
Consider on $\bbW$ the seminorms
$\lVert t\rVert_{b'}=
\int_0^{b'}\lVert t_{\rho_2}\rVert_{L^1}\,\rho_2\,d\rho_2$\,. 
If $\rho>0$ is fixed, $b'\ge \rho+b$, it follows from (a) that
$t\mapsto(s^{(2)}\sharp\,t)_\rho$ is continuous for $\lVert\ \rVert_{b'}$.
Then for each $b'>0$, we get that $t\mapsto s^{(2)}\sharp\,t$ is
a continuous mapping
$(\bbW,\lVert\ \rVert_{b'+b})\to(\bbW,\lVert\ \rVert_{b'})$.
As in Lemma\,\ref{wsub}, $\Cal B_1^\oplus$ is dense in
$(\bbW,\lVert\ \rVert_{b'})$ for each $b'>0$.
It follows by approximation that (c) holds pointwise (for
each $\rho>0$) for general $t\in\bbW$\,.
\\
The last statement of (d) follows from the Corollary to Lemma\,\ref{integk},
the others are easy.
\end{proof}\noindent
From now we assume that $w(\rho)\le\rho$ for all $\rho>0$ and that
$\frac{\rho}{w(\rho)}$ is bounded on $]0,b]$ for each $b>0$
\,(see Lemma\,\ref{wsub}). We define
\begin{multline*}
\Cal W=\{t\in\bW\!: \ s\sharp t\in\bW \text{ for all }
s\in \Cal W_0\cap\Cal B_1^\oplus \text{ \ and there exists }c\ge0
\\ \text{ such that }
\lVert s\sharp t\rVert_{\bW}\le c\,\lVert s\rVert_{\Cal B_1^\oplus}
\text{ for all } s\in\Cal W_0\cap\Cal B_1^\oplus\,\}\qquad\text{and}
\end{multline*}
\hspace{2mm}$\lVert t\rVert_{\Cal W}=
\sup\,\{\lVert \alpha t+s\sharp t\rVert_{\bW}:\ \alpha\ge0,
\ s\in\Cal W_0\cap\Cal B_1^\oplus,\
\alpha+\lVert s\rVert_{\Cal B_1^\oplus}\le1\}$ \;(see also Remark\,(c) below).
\begin{Lem}	\label{banm}	
Dual convolution $\sharp$ extends to an action of
$\Cal B_1^\oplus(\widehat{M(2)})$ on $\Cal W$\,. $\Cal W$ is a
Banach $\Cal B_1^\oplus$-module and ${\bf M}_n$-invariant,
$\lVert{\bf M}_n(t)\rVert_{\Cal W}=\lVert t\rVert_{\Cal W}\
\;(n\in\Z,\:t\in\Cal W)$\,.
\end{Lem}
\begin{proof}
It follows immediately from the definition that $\Cal W$ is a subspace of
$\bW$\,, $\lVert\ \rVert_{\Cal W}$ a norm,
$\lVert t\rVert_{\bW}\le\lVert t\rVert_{\Cal W}$\,,
$\lVert s\sharp t\rVert_{\bW}\le
\lVert s\rVert_{\Cal B_1^\oplus}\,\lVert t\rVert_{\Cal W}$ \,for
$s\in\Cal W_0\cap\Cal B_1^\oplus\,,\ t\in\Cal W$\,. Let $(t^{(n)})$
be a sequence in $\Cal W$ with $\lVert t^{(n)}\rVert_{\Cal W}\le1$ for
all $n$ and converging in
$\bW\,,\ t^{(0)}=\bW\text{-}\lim_{n\to\infty}t^{(n)}$\,.
Let $s\in\Cal W_0\cap\Cal B_1^\oplus,\ b'>0$ be fixed. It follows from
Lemma\,\ref{dcext}(a) (use that
$\frac{\rho}{w(\rho)}$ is bounded on bounded intervals) that there
exists $c>0$ (depending only on $s,b',w$) such that
$\int_0^{b'}\lVert(s\sharp t)_{\rho}\rVert_{L^1}\,w(\rho)\,d\rho\le
c\,\lVert t\rVert_{\bW}$ \,for all $t\in\bW$\,. This gives
$\int_0^{b'}\lVert(\,s\sharp(t^{(n)}-t^{(0)})\,)_{\rho}\rVert_{L^1}
\,w(\rho)\,d\rho\,\to\,0$ for $n\to\infty$\,, hence
$\int_0^{b'}\lVert(s\sharp t^{(0)})_{\rho}\rVert_{L^1}\,w(\rho)\,d\rho
\linebreak\le\lVert s\rVert_{\Cal B_1^\oplus}$\,. It follows that
$\lVert s\sharp t^{(0)}\rVert_{\bW}\le\lVert s\rVert_{\Cal B_1^\oplus}$ for
all $s\in\Cal W_0\cap\Cal B_1^\oplus$, thus $t^{(0)}\in\Cal W$\,.
If, in addition, $(t^{(n)})$ is a Cauchy sequence in $\Cal W$\,, we get
similarly that $t^{(n)}\to t^{(0)}$ in $\Cal W$\,, hence $\Cal W$ is
complete.
\\
It follows easily from Lemma\,\ref{dcext}(c) that $s\sharp t\in\Cal W$
for $s\in\Cal W_0\cap\Cal B_1^\oplus\,,\ t\in\Cal W$ and
$\lVert s\sharp t\rVert_{\Cal W}\le
\lVert s\rVert_{\Cal B_1^\oplus}\,\lVert t\rVert_{\Cal W}$\,.
By density (Lemma\,\ref{wsub}) $\sharp$ extends to $s\in\Cal B_1^\oplus$
and $\Cal W$ becomes a Banach $\Cal B_1^\oplus$-module. Concerning ${\bf M}_n$
the properties follow from Lemma\,\ref{dcext}(d).
\end{proof}\noindent
$\Cal W$ is the maximal $\Cal B_1^\oplus$-submodule of $\bW$
(for $\sharp$). By Lemma\,\ref{wsub}, $\Cal B_1^\oplus\subseteq\Cal W$
but the abstract definition does not give more information about the size of
$\Cal W$ \,(which depends also on $w$).
\begin{Pro}	 \label{wspec}	
For \,$w(\rho)=\dfrac\rho{1+\rho^2}$\,, we have
\;$L^1(]0,\infty[,d\rho,\Cal S_1(\Cal H))\subseteq\Cal W$\,.
\end{Pro}\noindent
Since $\widehat D=(\pi_\rho(D))$ maps $\Cal B_1^\oplus(\widehat{M(2)},\nu)$ to
$L^1(]0,\infty[,d\rho,\Cal S_1(\Cal H))$, we have found the desired
$\Cal B_1^\oplus$-module containing the image of the derivation. This will
finish our proof that $A(M(2))$ is not weakly amenable.
\begin{proof}
By direct evaluation, 
\;${\displaystyle\int}_{\!\!\!\lvert\rho_1-\rho_2\rvert}^{\rho_1+\rho_2}
K(\rho_1,\rho_2,\rho)\,w(\rho)\,d\rho=
\frac1\pi{\displaystyle\int}_{\!\!\!0}^\pi\dfrac1{1+\rho^2}\,d\mu=\linebreak
\dfrac1{\sqrt{(1+(\rho_1-\rho_2)^2)(1+(\rho_1+\rho_2)^2)}}\;\le\,
\dfrac1{\rho_2}$\,. \ Since
\\[2mm]
$\lVert (s\sharp t)_\rho\rVert_{L^1}\le\;
\iint_{D_\rho}\frac12\,
\lVert \theta(s_{\rho_1},t_{\rho_2})(\rho)\rVert_{L^1}
\,K(\rho_1,\rho_2,\rho)\,
\rho_1\rho_2\,d(\rho_1,\rho_2)\ $
\\[1mm]\hspace*{2cm} $\le\;
 \iint_{D_\rho} \lVert s_{\rho_1}\rVert_1\,\lVert t_{\rho_2}\rVert_1
\,K(\rho_1,\rho_2,\rho)\,
\rho_1\rho_2\,d(\rho_1,\rho_2)$, \ this gives
\\[.5mm]
$\int_0^\infty\lVert (s\sharp t)_\rho\rVert_{L^1}w(\rho)\,d\rho\le
\int_0^\infty\int_0^\infty
\lVert s_{\rho_1}\rVert_1\,\lVert t_{\rho_2}\rVert_1\,\rho_1\,d(\rho_1,\rho_2)
\le\lVert s\rVert_{\Cal B_1^\oplus}\,\int_0^\infty
\lVert t_{\rho_2}\rVert_1\,d\rho_2$\,.
\end{proof}
\begin{Rems}
(a) By additional estimates, one can show that more generally the conclusion
of Proposition\,\ref{wspec} holds for a (positive, measurable) weight function
$w$ on $]0,\infty[$\,, if $w$ is bounded and
$\int_0^\infty\frac{w(\rho)}\rho\;d\rho<\infty$\,. These conditions are
necessary as well. For the extension of $\sharp$ we needed that
$\frac\rho{w(\rho)}$ is bounded on every bounded interval $]0,b],\ b>0$.
If one wants that $\Cal B_1^\oplus\subseteq\Cal W$\,, then $\frac{w(\rho)}\rho$
should be bounded (on $]0,\infty[$) as well.
\item[(b)] Further derivations $\Cal B_1^\oplus\to\Cal W$ are obtained if
the distribution $d$ is
defined using some other direction from the normal subgroup $\R^2$. But,
again this does not work if $d$ involves a direction from the non-normal
factor.
\item[(c)] We used the unitization of $\Cal B_1^\oplus$ to define the norm on
$\Cal W$\,. One can show (using an appropriate approximate unit for
$\Cal B_1^\oplus$) that it  is also given by
$\lVert t\rVert_{\Cal W}=\sup\,\{\lVert s\sharp t\rVert_{\bW}:
\ s\in\Cal W_0\cap\Cal B_1^\oplus,\ \lVert s\rVert_{\Cal B_1^\oplus}\le1\}$\,.
\end{Rems}
\begin{Ex}
We consider the trigonometric basis $\chi_n(\psi)=e^{in\psi}\ \,(n\in\Z)$
of $\Cal H$\,. For $m,n\in\Z$ there are one dimensional operators
$A_{(m,n)}h=(h\vert\chi_n)\,\chi_m \;(h\in\Cal H)$. Then
$A_{(m,n)}\in\Cal S_1(\Cal H)$ with
$k_{A_{(m,n)}}(\psi,\psi_1)=\chi_m(\psi)\,\overline{\chi_n(\psi_1)}$\,.
For the corresponding coefficients
$u(x)=\tr(A_{(m,n)}\,\pi_{\rho}(x)^*)\ \,(x\in M(2))$ one gets
$u(\alpha x\alpha')=u(x)\,\chi_m(\alpha)\,\chi_n(\alpha')$ for
$\alpha,\alpha'\in\T$ and
$u(v)=e^{i(m-n)\phi}J_{n-m}(\rho r)$ for
$v=r\,\binom{\cos\phi}{\sin\phi}\in\R^2$ where $J_{n-m}$ denotes the
classical (J-) Bessel functions (see also \cite{V}\,p.\,206\,(9)\,). Then
Lemma\,\ref{integk} implies
$\theta_0(A_{(m_1,n_1)},A_{(m_2,n_2)},\rho_1,\rho_2)(\mu)=
\overline{k_{A_{(m_1,n_1)}}(\varphi,\varphi)}\,
k_{A_{(m_2,n_2)}}(\mu-\varphi,\mu-\varphi)\,A_{(m_1+m_2,n_1+n_2)}$
and $\theta(\dots)(\mu)=
2\Re\bigl(\overline{k_{A_{(m_1,n_1)}}(\varphi,\varphi)}\,
k_{A_{(m_2,n_2)}}(\mu-\varphi,\mu-\varphi)\bigr)\,A_{(m_1+m_2,n_1+n_2)}$
\;(this is an example where $\theta(A_1,A_2)(\mu)$ is one dimensional).
For $m_1=m_2=0\,,\ n=n_1+n_2$ one gets from Cor.\,\ref{prrh} to
Prop.\,\ref{tenseu} the multiplication formula
\[J_{n-n_2}(\rho_1r)J_{n_2}(\rho_2r)=
\frac2\pi\,\int_{\lvert\rho_1-\rho_2\rvert}^{\rho_1+\rho_2}
\frac{\cos(n\varphi-n_2\mu)\,\rho\,J_n(\rho r)}
{\sqrt{4\rho_1^2\rho_2^2-(\rho^2-\rho_1^2-\rho_2^2)^2}}\;d\rho
\]
for $r,\rho_1,\rho_2>0\,,\ n,n_2\in\Z$ with
$\mu=\arccos(\frac{\rho^2-\rho_1^2-\rho_2^2}{2\rho_1\rho_2})\,,\
\varphi=\arccos(\frac{\rho_1+\rho_2\cos\mu}\rho)$
\;(taking the principal branches of $\arccos$\,, i.e. $\mu,\varphi\in[0,\pi]$;
in fact the formula extends to $r\in\R$\,). Observe that there is an error in
\cite{V}, the corresponding formula (3) in \cite{V}\,p.\,210 is not correct
for all $m,n$\,.
\\
For $m_j=n_j=0$ (this corresponds to the ``radial" part of
$A(M(2))\,,\,B(M(2))$, i.e. thefunctions that are both left and right
$\T$-periodic) one gets the multiplication of the Bessel-Kingman hypergroup,
\cite{BH}\,p.\,235 with $\alpha=0$.
\end{Ex}
\vspace{.5mm}
\section{The groups $M(2)_{\Cal K}$ and $\widetilde{M(2)}$}\vspace{1mm}
\label{euclK}
$\widetilde{M(2)}$ is covered by \cite{LLSS}\,Th.\,1.6  (using
spectral synthesis). $M(2)_{\Cal K}$ was the main open case for our Theorem
(see \cite{LLSS}\,Ex.\,3.4).
\\
We take up the notations of Section\,\ref{mini},
$M(2)_{\Cal K}=\R^2\rtimes\Cal K$\,, $\Cal K$ is a non-trivial
compact group and we have given a continuous homomorphism
$\varphi_0\!\!:\R\to \Cal K$ with dense image, $\Cal K$ is written additively
and there exists a proper closed subgroup $\Cal K_0$ of  $\Cal K$ such that
$\varphi_0(2\pi\Z)=\Cal K_0\cap\varphi_0(\R)$\,. Then
$\Cal K/\Cal K_0\cong\T\ (=\R/2\pi\Z)$ giving a continuous surjective
homomorphism \,$\alpha_{\Cal K}\!:\Cal K\to\T$ with
$\ker\alpha_{\Cal K}=\Cal K_0\,,\ \alpha_{\Cal K}\circ\varphi_0$ is the
canonical projection $\R\to\T$\,. The action of $\Cal K$ on $\R^2$ is
given by $k\circ v=A_{\alpha_{\Cal K}(k)}v\ \;(k\in\Cal K\,,\ v\in\R^2)$,
using the rotations $A_\alpha$ from Section\,\ref{euclm}. Then
$\alpha_{\Cal K}$ extends to a continuous surjective homomorphism
$M(2)_{\Cal K}\to M(2)$\,, again denoted by $\alpha_{\Cal K}$\,, with
$\ker\alpha_{\Cal K}=\Cal K_0$\,. For Haar
measure we take $\dfrac1{2\pi}\,dv\,dk$ \,(with normalized Haar measure on
$\Cal K$\,). $\He_{\Cal K}$ is unimodular.
As in Section\,\ref{heisK}, the dual group
$\Gamma=\widehat{\Cal K}$ is considered as a subgroup of $\R$ in such a
way that for $t\in\Gamma$ the corresponding character $\chi_t$ of $\Cal K$
satisfies $\chi_t(\varphi_0(z))=e^{itz}$ for $z\in\R$ \,(see Rem.\,(b) in
Section\,\ref{mini} and Rem.\,(c) in Appendix\,\ref{dense}). Under
this identification $\Cal K_0^\perp=\Z$\, (hence $\Z\subseteq\Gamma$).
Conversely, for any subgroup $\Gamma$ of $\R$ with $\Gamma\supset\Z$\,,
$\Cal K=\widehat\Gamma$ gives a compact group satifying the conditions
above (with $\varphi_0$ the dual of the inclusion $\Gamma\to\R$\,, i.e.
restricting the characters of $\R$ to $\Gamma$\,; $\Cal K_0=\Z^\perp$;
\,$\alpha_{\Cal K}$\,: restricting the characters of $\Gamma$ to $\Z$\,).
\\[1mm]
$\widetilde{M(2)}=\R^2\rtimes\R$ (universal covering group of $M(2)$).
For Haar measure we take $\dfrac1{(2\pi)^2}\,dv\,d\alpha$ \,(now
$\alpha\in\R$).
$\widetilde{M(2)}$ is unimodular. $\varphi_0\!\!:\R\to \Cal K$ extends
to a continuous homomorphism $\varphi\!:\widetilde{M(2)}\to M(2)_{\Cal K}$
with dense image. In particular for $\Cal K=\T=\R/2\pi\Z$ \,($\varphi_0$ the
quotient mapping), $\varphi_1\!:\widetilde{M(2)}\to M(2)$ gives the
covering homomorphism, for general
$\Cal K\ \;\varphi_1=\alpha_{\Cal K}\circ\varphi$\,.
The representations $\pi_\rho$ of $M(2)$ (Section\,\ref{euclm}) lift
to $\widetilde{M(2)}$ and
$M(2)_{\Cal K}$\,, we denote these again by $\pi_\rho\ \,(\rho>0)$.
$\R^2$ is the commutator subgroup of $\widetilde{M(2)}$\,, hence the
one dimensional representations of $\widetilde{M(2)}$ are given by
the characters $\{\chi_t:t\in\R\}$ of $\R\cong\widetilde{M(2)}/\R^2$.
$2\pi\Z\;(\subseteq\R)$ is the centre of $\widetilde{M(2)}$\,. If
$\pi$ is an irreducible representation of $\widetilde{M(2)}$ its
restriction to $2\pi\Z$ is given by a character which extends to a
character $\chi_t$ of $\R$ (for some $t\in\R$). Then
$\overline{\chi_t}\otimes\pi$ is trivial on $2\pi\Z$\,, hence it induces
an irreducible representation of $M(2)\cong\widetilde{M(2)}/(2\pi\Z)$.
If $\pi$ is infinite dimensional, it follows that
$\overline{\chi_t}\otimes\pi\simeq\pi_\rho$\,, hence
$\pi\simeq\chi_t\otimes\pi_\rho$ for some $\rho>0$\,. Put
$\Cal R=[0,1[$\,. $\chi_t$ and $\chi_{t+n}\ (n\in\Z)$ coincide on
$2\pi\Z$\,, hence we may choose $t\in\Cal R$\,. We put
$\brho=(t,\rho)$ and use the notation
$\pi_{\brho}=\chi_t\otimes\pi_\rho$ \,(considered again as a representation on
$\Cal H$\,, see the proof of Proposition\,\ref{tenseu}). It is not hard to see that
for $\brho\neq\brho'\in\Cal R\times]0,\infty[$ the representations
$\pi_{\brho},\pi_{\brho'}$ are not equivalent. Thus we get
\[
\Bigl(\widetilde{M(2)}\Bigr)\sphat=
\{\chi_t:t\in\R\}\cup\{\chi_t\otimes\pi_\rho:\,0\le t<1,\ \rho>0\}
\]\noindent
(as a complete set of representatives; one might get this also from
Mackey's theory). One can show that $\{\pi_\brho\}$
is open and (taking on $\Cal R$ the quotient topology from $\R/\Z$) it carries
the product topology of $\Cal R\times]0,\infty[$\,. On $\{\chi_t\}$ one gets
the standard topology of $\R$\,,
\;$\overline{\{\chi_t\otimes\pi_\rho\}}=
\{\chi_t\otimes\pi_\rho\}\cup\{\chi_{t+n}: n\in\Z\}$ for each
$\rho>0\,,\ t\in\Cal R$\,.
\\[2mm]
For $M(2)_{\Cal K}$ the centre is $\Cal K_0$\,. Let $\Cal R_\Gamma$ be a set
of representatives for the cosets in $\Gamma/\Z$ \,(e.g.
$\Cal R_\Gamma=\Gamma\cap[0,1[$). Then a similar argument as above gives
\[
(M(2)_{\Cal K})\sphat=
\{\chi_t:t\in\Gamma\}\cup
\{\chi_t\otimes\pi_\rho:\,t\in\Cal R_\Gamma\,,\ \rho>0\}
\]\noindent
On the open subset $\{\pi_\brho\}$ one has the product topology of
$\Cal R_\Gamma\times]0,\infty[$ (taking the discrete topology on
$\Cal R_\Gamma$) and on $\{\chi_t\}$ the discrete topology. The closure
$\overline{\{\chi_t\otimes\pi_\rho\}}$ is as above.
\begin{Lem}   \label{pleuclK}	   
\thetag{a}\ Using Haar measure and parametrization as above, the Plancherel
measure
$\bnu$ of $\widetilde{M(2)}$ is concentrated on
$\{\chi_t\otimes\pi_\rho:\,0\le t<1,\ \rho>0\}$. It is the product of
standard Lebesgue measure on $[0,1[$ and 
\,$d\nu=\rho\,d\rho$ on $]0,\infty[$\,.\vspace{1mm}
\item[(b)]
Plancherel measure
$\bnu$ of $M(2)_{\Cal K}$ is concentrated on
$\{\chi_t\otimes\pi_\rho:\,t\in\Cal R_\Gamma\,,\ \rho>0\}$. It is the product
of counting measure on $\Cal R_\Gamma$ and 
\,$d\nu=\rho\,d\rho$ on $]0,\infty[$\,.
\end{Lem}\noindent
If $\Cal K$ is not metrizable (equivalently, if $\Gamma$ is uncountable)
(b) gives again an example of the Plancherel
theorem for a group that is non-compact, $\sigma$-compact, but not second
countable (see also Rem.\,(d) in Appendix\,\ref{AppenB}).
\begin{proof}
Again, this follows by direct computations (compare Lemma\,\ref{plheisK}),
using the Plancherel theorem for $\R^2,\R$ and $\Cal K$\,.
\end{proof}\noindent
The distribution $d$ and the corresponding derivation $D$ can be defined as
in Section\,\ref{euclm}. $\widetilde{M(2)}$ is a Lie group, 
$d$ and $D$ are defined on $C^1(\widetilde{M(2)})$. For $M(2)_{\Cal K}$
we can use approximation by Lie groups as in Section\,\ref{heisK},
$d$ and $D$ are defined on $\bigcup\{C^1(M(2)_{\Cal K/K_1})\}$ ranging over
the closed subgroups $K_1$ of $\Cal K_0$ such that $\Cal K_0/K_1$ is Lie
(equivalently, $K_1^\perp$ is finitely generated and $K_1^\perp\supseteq\Z$).
We have $\chi_t(d)=0$ and for $\brho=(t,\rho)$ this gives
$\pi_\brho(d)=\pi_\rho(d)$ where again
$(\pi_\rho(d)g)(\psi)=-i\rho\cos(\psi)\,g(\psi)$
\,(for $g\in\Cal H=L^2(\R,\frac1{2\pi}d\psi)$\,).
In both cases $D$ does {\it not} produce a bounded mapping
$A\to\VN$\,. We will next investigate the dual convolution and then
use the construction for $M(2)$ to find a suitable $\Cal B_1^\oplus$-module
$\Cal W$\,.
\\[3mm]
First, we consider now $M(2)_{\Cal K}$\,. We assume that
$\Cal R_\Gamma=\Gamma\cap[0,1[$\,. We start with the decomposition of tensor
products making a similar procedure as in Section\,\ref{euclm}.
For $t\in\R\,,\ [t]$ denotes the integer
part, $\{t\}$ the fractional part. For
$\brho_j=(t_j,\rho_j)\in\Cal R_\Gamma\times]0,\infty[\ \,(j=1,2)$ we define
$(W_{\!\brho_1,\brho_2}f)(\mu)=
M_{-[t_1+t_2]}\circ((W_{\!\rho_1,\rho_2}f)(\mu))$\linebreak
\;($W_{\!\brho_1,\brho_2}\!:\Cal H\otimes_2\Cal H\to
L^2(\T,\frac1{2\pi}d\mu,\Cal H)$\,). With the notations made before
Proposition\,\ref{tenseu} we get
\begin{Pro} \label{tenseuK}  
For $\brho_j=(t_j,\rho_j)\in\Cal R_\Gamma\times]0,\infty[\ \,(j=1,2),\
x\in M(2)_{\Cal K}$\,, we have\vspace{-1.5mm}
\begin{gather*}
W_{\!\brho_1,\brho_2}\circ(\pi_{\brho_1}\otimes\pi_{\brho_2})(x)
\circ W_{\!\brho_1,\brho_2}^*=
\int_\T^\oplus\pi_{(\{t_1+t_2\},\rho)}(x)\,d\mu\ .
\end{gather*}
\end{Pro}\noindent
\begin{proof}
Remembering our identifications, we have
\,$(\pi_{\brho_1}\otimes\pi_{\brho_2})(x)=
\chi_{t_1+t_2}(x){\ \;}\linebreak
(\pi_{\rho_1}\otimes\pi_{\rho_2})(x)$\,. Using
$t=[t]+\{t\}$\,, the result can be obtained from Proposition\,\ref{tenseu}.
\end{proof}\noindent
We put
$\btheta_0(A_1,A_2,\brho_1,\brho_2)=
\pr_{\Cal S}(W_{\!\brho_1,\brho_2}\circ(A_1\otimes A_2)
\circ W_{\!\brho_1,\brho_2}^*)$
and
$\btheta(A_1,A_2)(\rho)=\btheta_0(A_1,A_2)(\mu)+\btheta_0(A_1,A_2)(2\pi-\mu)\
\ (A_1,A_2\in\Cal S^1(\Cal H)\,)$.
\setcounter{Cor1}{0}
\begin{Cor1}	 \label{prKrh}	   
For
$A_1,A_2\in\Cal S^1(\Cal H),\
\brho_j=(t_j,\rho_j)\in\Cal R_\Gamma\times]0,\infty[$ consider coefficients
\,$u_j(x)=\tr(A_j\,\pi_{\brho_j}(x)^*)\ \ (j=1,2,\ x\in M(2)_{\Cal K})$. Then
\;$\btheta(A_1,A_2,\brho_1,\brho_2)
={\bf M}_{-[t_1+t_2]}(\theta(A_1,A_2,\rho_1,\rho_2))\ \in
L^1(\,[\lvert\rho_1-\rho_2\rvert,\rho_1+\rho_2]\,,
\frac12K(\rho_1,\rho_2,\rho)\,\rho\,d\rho\,,\Cal S^1(\Cal H)\,)$ \\
\  and
\[ u_1(x)\,u_2(x)=
\,\int_{\lvert\rho_1-\rho_2\rvert}^{\rho_1+\rho_2}
\frac12\,\tr\bigl(\,\pi_{(\{t_1+t_2\},\rho)}(x)^*\,
\btheta(A_1,A_2)(\rho)\,\bigr)
\,K(\rho_1,\rho_2,\rho)\,d\nu(\rho)\,.
\]
\end{Cor1}\noindent
In particular,
\;$\pi_{\brho_1}\otimes\pi_{\brho_2}\simeq
\,2\,\int_{\lvert\rho_1-\rho_2\rvert}^{\rho_1+\rho_2}
\pi_{(\{t_1+t_2\},\rho)}\,d\rho$\,. Similarly, there is an
analogue of Corollary\,\ref{prmu} to Proposition\,\ref{tenseu} in terms
of $\btheta_0$\,. 
\begin{proof}
An expression for $\btheta_0$ and then for $\btheta$ can be obtained as
in Lemma\,\ref{integk}. For the product formula compare the
Corollaries\,\ref{prmu},\ref{prrh} to Proposition\,\ref{tenseu}, using
now Proposition\,\ref{tenseuK}.
\end{proof}\noindent
For an operator field
$T=(T_\brho)\in\Cal B_1^\oplus((M(2)_{\Cal K})\sphat\,,\bnu)$ and
$t\in\Cal R_\Gamma$ we put $T^{(t)}=(T_{(t,\rho)})_{\rho>0}$\,.
Then \,$T^{(t)}\in\Cal B_1^\oplus(\widehat{M(2)} ,\nu)$ \ and
$\lVert T\rVert_{\Cal B_1^\oplus}=
\sum_{t\in\Cal R_\Gamma}\lVert T^{(t)}\rVert_{\Cal B_1^\oplus}$\,.
Thus we get an isomorphism
$\Cal B_1^\oplus((M(2)_{\Cal K})\sphat\,,\bnu)\,\cong\,
l^1(\Cal R_\Gamma\,,\Cal B_1^\oplus(\widehat{M(2)},\nu)\,)$.
\begin{Cor1}	 \label{dualcK}	   
For
\;$S,T\in \Cal B_1^\oplus((M(2)_{\Cal K})\sphat\,,\bnu),\
t\in\Cal R_\Gamma$\,, we have
\[
(S\sharp T)^{(t)}=\sum_{t_1\in\Cal R_\Gamma}
{\bf M}_{[t-t_1]}(S^{(t_1)}\sharp T^{(\{t-t_1\})})
\]
\end{Cor1}\noindent
Thus one gets a sort of twisted convolution.
\begin{proof}
This follows from Corollary\,\ref{dualc} to Proposition\,\ref{tenseu} and
Corollary\,\ref{prKrh} above.
\end{proof}\noindent
Let $\Cal W(M(2))$ be the Banach module defined before Lemma\,\ref{banm}
with $w$ as in Proposition\,\ref{wspec} (see also Remark (a) in
Section;\ref{euclm}). We consider now
\[
\Cal W=l^1(\Cal R_\Gamma\,,\Cal W(M(2))\,)
=\{(w^{(t)}): \;w^{(t)}\in \Cal W(M(2)),\
\sum_{t\in\Cal R_\Gamma}\lVert w^{(t)}\rVert_{\Cal W}<\infty\}\,.
\]
\begin{Pro} \label{wK}  
We have \,$\Cal B_1^\oplus((M(2)_{\Cal K})\sphat\,,\bnu)\subseteq\Cal W$ and
\,$l^1(\Cal R_\Gamma\,,L^1(]0,\infty[\,,d\rho\,,\linebreak\Cal S_1(\Cal H))\,)
\subseteq\Cal W$\,. Dual convolution extends to an action of $\Cal B_1^\oplus$
on $\Cal W$\,. $\Cal W$ is a Banach\linebreak $\Cal B_1^\oplus$-module.
\end{Pro}\noindent
This finishes our proof that $A(M(2)_{\Cal K})$ is not weakly amenable.
\begin{proof}
This follows from the properties of $\Cal W(M(2))$ obtained in
Lemma\,\ref{banm}, Proposition\,\ref{wspec} and Lemma\,\ref{wsub}, and
using Corollary\,\ref{dualcK} above and Lemma\,\ref{dcext}\,(d) to
get the extension of $\sharp$\,.\vspace{1mm}
\end{proof}\noindent
The construction works quite similarly for $\widetilde{M(2)}$\,.
With $\Cal R=[0,1[$\,, Proposition\,\ref{tenseuK} holds in the same way
for $\brho_j\in\Cal R\times]0,\infty[\,,\
x\in\widetilde{M(2)}$ and also Corollary\,\ref{prKrh}. As above,
\;$\Cal B_1^\oplus(\,\bigl(\widetilde{M(2)}\bigr)\sphat\,,\bnu)\cong
L^1([0,1[,dt,\Cal B_1^\oplus(\widehat{M(2)},\nu)\,)$ by transforming operator
fields. Similarly, in Corollary\,\ref{dualcK} one gets
$(S\sharp T)^{(t)}=\int_{\Cal R}
{\bf M}_{[t-t_1]}(S^{(t_1)}\sharp T^{(\{t-t_1\})})\,dt_1$\,.
Finally, one can take \,$\Cal W=L^1(\Cal R\,,dt,\Cal W(M(2))\,)$\,.
\begin{Rems}
The formulas given above depend on the choice of $\Cal R_\Gamma$\,.
If $\Cal R_\Gamma\;(\subseteq\Gamma)$ is an arbitrary set of representatives
for $\Gamma/\Z$ there is a section $\sigma\!:\Gamma\to\Cal R_\Gamma$ and
a corresonding cocycle $n(t_1,t_2)=\sigma(t_1)+\sigma(t_2)-\sigma(t_1+t_2)\
(\in\Z \text{ for }t_1,t_2\in\Gamma)$. Then in the definition of
$W_{\!\brho_1,\brho_2}$ one has to replace $[t_1+t_2]$ by $n(t_1,t_2)$
and in Proposition\,\ref{tenseuK} replace $\{t_1+t_2\}$ by $\sigma(t_1+t_2)$,
the same in Corollary\,\ref{prKrh}. In Corollary\,\ref{dualcK} one replaces
$\{t-t_1\}$ by $\sigma(t-t_1)$ and $[t-t_1]$ by $-n(t_1,t-t_1)$.
\end{Rems}
\section{The Gr\'elaud groups $\G_a$}\vspace{1mm} \label{gr}
They were treated in \cite{LLSS}\,sec.\,2.3, using spectral synthesis.
We present also (slightly more briefly) our method for this case.
\\
We use the complex presentation $\G_a=\C\rtimes\R$ with the action
$t\circ z=e^{wt}z \linebreak(t\in\R,\,z\in\C)$,
putting $w=a+i$ for $a\neq0$\,.
It is not hard to see that $\G_{-a}\cong\G_a$ hence we will always
assume that $a>0$ (similarly for the version used in \cite{FL}\,4.4,
$\G_{-a}\cong G_\theta$ with $\theta=-\frac1a$). Using the homeomorphism with
$\C\times\R$ given by
$(z,t)\mapsto zt$ we take for (left) Haar measure $e^{-2at}dz\,dt$\,.
$\G_a$ is not unimodular, $\Delta(zt)=e^{-2at}$.
\\
The irreducible representations can again be obtained by Mackey's theory.
We take here $\Cal H=L^2(\R)$ (as in Section\,\ref{heis}). For real $b\neq0$
an irreducible representation $\pi_b$ is defined for $f\in\Cal H$ by
\,$(\pi_b(z)f)(\psi)=\exp{(ib\,\Re(e^{-w\psi}z))}f(\psi)\,,\
(\pi_b(t)f)(\psi)=f(\psi-t)\ \linebreak(\Re$ denoting the real part;
in the notaion of \cite{FL} this corresonds to $\pi_\mu$ with $\mu=-b$).
For $t=2\pi\in\R\,,\ \;\pi_b(t)\pi_b(\bx)\pi_b(-t)=
\pi_{be^{2\pi a}}(\bx)\ \,(\bx\in\G_a)$ gives
$\pi_b\simeq\pi_{b\,e^{2\pi a}}$ \,and for $t=\pi\,,\ 
\,\pi_b(t)\pi_b(\bx)\pi_b(-t)=\pi_{-be^{\pi a}}(\bx)$ implies
$\pi_b\simeq\pi_{-b\,e^{\pi a}}$\,. It follows that one can restrict to
$b\in[1,e^{2\pi a}[$ and in this domain one can show that the representations
are pairwise non-equivalent, they exhaust (up to equivalence) all the infinite
dimensional irreducible representations of $\G_a$\,.
$\{\pi_b:b\in[1,e^{2\pi a}[\,\}$ defines an open, compact (but non-closed)
subset of $\widehat{\G_a}$ \;(again, the complement consists of the one
dimensional representations, they are given by $\widehat{\R}$\,). It
carries the quotient topology of $[1,e^{2\pi a}]$\,, obtained by
identifying the end points. Formulas for closures are as in
Section\,\ref{euclm}.
\\
$\G_a$ is type I, Plancherel measure is concentrated on
$\{\pi_b:b\in[1,e^{2\pi a}[\,\}$. Put $(Kf)(\psi)=e^{-2a\psi}f(\psi)$ (densely
defined on $\{f\in\Cal H:\,Kf\in\Cal H\}$\,). $K$ is positive definite,
self-adjoint and satisfies the semi-invariance relations for all $\pi_b$\,.
Hence it gives Duflo--Moore operators (independent of $b$\,, see
Appendix\,\ref{AppenB}). The corresponding Plancherel measure is given
by $d\nu=\frac1{(2\pi)^2}b\,db$ \,(that can again be verified by direct
calculation). Thus
$\Cal F_A(f)=([\pi_b(f)K])\ ([\cdot]$ denoting now the closure of
the operator), $\Cal B_1^\oplus(\widehat{\G_a},\nu)\cong
L^1([1,e^{2\pi a}[,\frac1{(2\pi)^2}b\,db,\Cal S_1(\Cal H))$,
general elements of
$\Cal B_1^\oplus(\widehat{\G_a},\nu)$ will be denoted as
\,$T=(T_b)_{1\le b<e^{2\pi a}}$\,.
\\[0mm plus 1mm]
As before, we consider distributions defined using directional derivatives
from the normal subgroup $\C$\,. For example, writing $z=x+iy$\,, put
$d=\lim_{x\to0}\frac1x(\delta_{-x}-\delta_{0})$ and
$D(f)=f\star d$\,, defined on $C^1(\G_a)$. Put $\widehat D_b=\pi_b(d)$
and for $T=(T_b)\,,\ \;\widehat D(T)=([T_b\widehat D_b])$.
Then 
$(\widehat D_bg)(\psi)=-ibe^{-a\psi}\cos(\psi)\,g(\psi)$
\,(for appropriate $g\in\Cal H$) and
\;$\Cal F_A(D(f))=\widehat D(\Cal F_A(f))$
for $f\in A(\G_a)\cap C^1(\G_a)$ with $D(f)\in A(\G_a)$\,. One can check
that $\{([T_b\widehat D_bK^{-1}]): T=(T_b)\in\Cal B_1^\oplus\}$
is not contained in
$L^\infty(\widehat{\G_a},\nu,\Cal B(\Cal H))$\,, hence
$D$ does not produce a bounded mapping $A(\G_a)\to\VN(\G_a)$\,. 
\\[2mm plus 1mm]
For the decomposition of tensor products, we proceed similarly as in
Section\,\ref{euclm}.
$\Cal H\otimes_2\Cal H$ is identified with
$L^2(\R^2)$. For $f\in\Cal H\otimes_2\Cal H$
define \,$(Wf)(\mu)(\psi)=f(\psi,\psi+\mu)$ giving an isometric isomorphism
$W\!:\Cal H\otimes\Cal H\to L^2(\R,\Cal H)$\,.
For $\varphi\in\R\,,\ g\in\Cal H$\,, we define
$(U_\varphi g)(\psi)=g(\psi-\varphi)$, giving a unitary operator on
$\Cal H$\,, depending strongly continuously on $\varphi$\,.  Fix now
$b_1,b_2\neq0$ and for $\mu\in\R$ put
$b_1+b_2\,e^{-w\mu}=r\,e^{-i\varphi}$ (polar coordinates) and
$b=r\,e^{a\varphi}$, defining functions
$r=r(b_1,b_2,\mu),\
\varphi=\varphi(b_1,b_2,\mu)$ (see below) and $b=b(b_1,b_2,\mu)$ \,. Finally,
we define
\,$(W_{\!b_1,b_2}f)(\mu)=U_\varphi((Wf)(\mu))$
\,($f\in\Cal H\otimes_2\Cal H$), giving again an isometric isomorphism
$W_{\!b_1,b_2}\!:{\Cal H\otimes_2\Cal H}\to L^2(\R,\Cal H)$,
\;$(b_1,b_2)\mapsto W_{\!b_1,b_2}$ is strongly continuous. As an analogue
of Proposition\,\ref{tenseu} one gets
$W_{\!b_1,b_2}\circ{(\pi_{b_1}\otimes\pi_{b_2})(\bx)}
\circ W_{\!b_1,b_2}^*=\int_\R^\oplus\pi_b(\bx)\,d\mu$ for $\bx\in\G_a$\,.
$\pr_{\Cal S}\!:\Cal S^1(L^2(\R,\Cal H))
\to L^1(\R,\Cal S^1(\Cal H))$ shall be dual to 
the embedding of the subalgebra $L^\infty(\R)\overline\otimes\Cal B(\Cal H)$
into $\Cal B(L^2(\R,\Cal H))$ and for $A_1,A_2\in\Cal S^1(\Cal H)$ we define
again
\;$\theta_0(A_1,A_2)=\pr_{\Cal S}(W_{\!b_1,b_2}\circ{(A_1\otimes A_2)}
\circ W_{\!b_1,b_2}^*) \quad (\in L^1(\R,\Cal S^1(\Cal H)\,)$. Then
for $u_j(\bx)=\tr(A_j\,\pi_{b_j}(\bx)^*)\ \;{(j=1,2,}\linebreak
\bx\!\in\!\G_a)$
it results as in Corollary\,\ref{prmu} to Proposition\,\ref{tenseu} that
\;$u_1(\bx)\,u_2(\bx)={\,}\linebreak
\int_\R\tr\bigl(\,\pi_b(\bx)^*\,\theta_0(A_1,A_2)(\mu)\,\bigr)\,d\mu$\,.
\\[0mm plus 1mm]
The definition of the function $b(\mu)$ and the operators
$W_{\!b_1,b_2},\!\theta_0$
depends on the choice of~$\varphi$\,. If $\varphi$ is replaced by
$\varphi+2\pi$\,, then $b$ transforms to $b\,e^{2\pi a}$ and similarly for
$W_{\!b_1,b_2},\theta_0$\,. The expression for $u_1(\bx)\,u_2(\bx)$ above is
valid for any
measurable choice of $\varphi(\mu)$\,. We will now assume that $\varphi$
is chosen so that $b\in[1,e^{2\pi a}[$
holds. This determines $\varphi$
uniquely for all $\mu$ where $r\neq0$\,. For fixed
$b_1,b_2\,, \linebreak r(b_1,b_2,\mu)=0$ can happen just for one $\mu$
\,(see Remark\,(b)).
We define now 
$\theta(A_1,A_2)(b)=\frac{(2\pi)^2}b\sum_{b(\mu)=b}
\theta_0(A_1,A_2)(\mu)\,\lvert\frac{db}{d\mu}(\mu)\rvert^{-1}$\,.
This gives \;$u_1(\bx)\,u_2(\bx)=\int_1^{e^{2\pi a}}
\tr\bigl(\,\pi_b(\bx)^*\,\theta(A_1,A_2)(b)\,\bigr)\,d\nu(b)$\,.
Finally, it follows as in Corollary\,\ref{dualc} of Proposition\,\ref{tenseu}
that for  $S=(S_b),T=(T_b)\in \Cal B_1^\oplus(\widehat{\G_a},\nu)$, we have
\[
(S\sharp T)_b\,=\iint\theta(S_{b_1},T_{b_2}\!)(b)\,d\nu(b_1)\,d\nu(b_2)\quad
\text{a.e.}
\]\noindent
As in Lemma\,\ref{integk} there is an expression in terms of
integration kernels, for $\mu\in\R$ \,one has
\,$k_{\theta_0(A_1,A_2)(\mu)}(\psi,\psi_1)\,=\,
k_{A_1}(\psi-\varphi\,,\psi_1-\varphi)\:
k_{A_2}(\psi-\varphi+\mu\,,\psi_1-\varphi+\mu)$ . It allows
extensions of $\theta_0\,,\theta$ and $\sharp$\,.
For a Borel measurable function $a_0\!:\R^2\to\C$ and $\xi_1,\xi_2\in\R$ we
put $\lVert a_0\rVert_{L^1\!,\xi_1,\xi_2}=
\int_{\xi_1}^{\xi_1+1}\int_0^\infty\lvert\,a_0(\xi+t,\xi_2+t)\rvert\,dt\,d\xi
\,.$
We consider now measurable fields $t=(t_b)_{1\le b<e^{2\pi a}}$ of functions,
i.e. $(b,\psi,\psi_1)\to t_b(\psi,\psi_1)$ shall be Borel measurable and
use
\;$\bbW=\{t:\int\lVert t_b\rVert_{L^1\!,\xi_1,\xi_2}db<\infty\text{ for all }
\xi_1,\xi_2\in\R\}$ \,(identifying fields which coincide a.e.),
\;$\Cal W_0=\{t:\;\sup_b\lVert t_b\rVert_\infty<\infty\text{ \,and there exists
a compact set}\linebreak C\subseteq\R^2 \text{ such that for all }b\
\;t_b(\psi,\psi_1)=0
\text{ holds when }(\psi,\psi_1)\notin C\}$. Then, using the formula above,
$s\sharp t$ can be defined for $s\in\Cal W_0\,,\;t\in\bbW$\,. The case
$b_2=e^{\pm\pi a}b_1$ causes some technical complications, it is not always
true that $s_1\sharp s_2\in\Cal W_0$ for $s_1,s_2\in\Cal W_0$\,.
\\
Finally, we consider
\;$\bW=\{t:\;\sup_{\xi_1,\xi_2}\frac1{1+e^{-\xi_2}}
\int\lVert t_b\rVert_{L^1\!,\xi_1,\xi_2}db\;<\;\infty\,\}$\,. Clearly
$\bW\subseteq\bbW$ and one can define now $\Cal W$ as in Section\,\ref{euclm}.
Similarly as in Lemma\,\ref{banm}, dual convolution $\sharp$ extends to an
action of
$\Cal B_1^\oplus(\widehat{\G_a})$ on $\Cal W$\,, making $\Cal W$ a
Banach $\Cal B_1^\oplus$-module. One can show that
\,$\{T:\int\lVert\,[T_bK^{\frac12}]\,\rVert_1\,d\nu(b)<\infty\,\}
\subseteq\Cal W$\,
hence $\widehat D(\Cal B_1^\oplus)\subseteq\Cal W$ finishing our proof for
$A(\G_a)$\,.
\begin{Rems}
(a) It follows from the arguments above that
\;$\pi_{b_1}\otimes\pi_{b_2}\simeq
\,\int_1^{e^{2\pi a}}n(b)\,\pi_b\,db$ for some multiplicities
$n(b,b_1,b_2)\in\N_0\cup\{\infty\}$\,.  The behaviour of $n(b)$ depends on
the relative sizes of $a,b_1,b_2$\,. For
$b_2\neq e^{\pm\pi a}b_1\,,\ 1\le b_1,b_2\le e^{2\pi a}$ one has that
$n(b)$ is finite for all $b\neq b_1,b_2$\,, but $n(b)\to\infty$ when
$b\to b_1$ or $b\to b_2$ \,(and $n(b)$ is bounded on every closed interval
not containing $b_1,b_2$). In the exceptional cases $b_2=e^{\pm\pi a}b_1$ one
gets $n(b)=\infty$ for all $b$\,.
\item[(b)]\ Geometrically, $\beta(\mu)=b_1+b_2\,e^{-w\mu}$ describes a
logarithmic spiral. For $1\le b_1,b_2\le e^{2\pi a},\ \ \beta(\mu)=0$
(equivalently $r(\mu)=0$) happens only for $b_2=e^{\pm\pi a}b_1$\,, and
for $b_2=e^{\pi a}b_1$ it happens only for $\mu=\pi$\,. With some
computations one can show that
\;$\frac{db}{d\mu}=-e^{a(\varphi-\mu)}b_1b_2(a^2+1)\frac{\sin(\mu)}r$ at
every continuity point of $\varphi$\,. It follows that $b(\mu)$ is
monotonic on every subinterval of $[k\pi,(k+1)\pi]$ where $\varphi$ is
chosen as a continuous function of $\mu\,\ (k\in\Z)$. This justifies also our
definition of $\theta$ above (in particular that
$\theta(A_1,A_2)(b)\in\Cal B_1^\oplus(\G_a)$\,), based on the coordinate
transformation $b=b(\mu)$\,.
\item[(c)]\ Our representations $\pi_b$ are basically those of \cite{FL}
(more pecisely, $\pi_b\simeq\pi_\mu$ with $\pi_\mu$ as in \cite{FL}\,p.108,
$\mu=-b$). Another version is given by
\,$(\pi'_\alpha(z)f)(\psi)=\exp{(i\,\Re(e^{-w\psi}z\alpha))}f(\psi)\,,\
(\pi'_\alpha(t)f)(\psi)=f(\psi-t)$ where
$\alpha\in\C\,,\;\lvert\alpha\rvert=1$ and as before $f\in\Cal H$\,.
Then $\pi'_\alpha\simeq\pi_b$ with $\alpha=b^{-\frac ia},\,b>0$. The family
$\{\pi'_\alpha:\lvert\alpha\rvert=1\}$ better reflects the periodicity
in the infinite dimensional representations of $\G_a$\,.
\end{Rems}\vspace{2mm}
\appendix
\section{Connected dense subgroups} \label{AppenA}
\vspace{1mm}\noindent
We elaborate (and slightly extend) results of Malcev, Goto, Poguntke.
\begin{ThmA}	\label{dense}
Let $G$ be a locally compact group, $H$ a connected Lie group,\linebreak
$\varphi\!:H\to G$ \,a %
continuous homomorphism such that $\ker\varphi$ is discrete.
Then there exists a closed vector subgroup $V$ of $H$ \ (i.e. $V\cong\R^n$ for
some $n\ge0$) such that 
$\overline{\varphi(H)}=\varphi(H)\:\overline{\varphi(V)}$\,,
$\varphi(V)$ is relatively compact in $G$\,,
$\overline{\varphi(V)}\cap\varphi(H)=\varphi(V)$ and the image $\iota_H(V)$
of $V$ in the  group
of inner automorphisms $\Int(H)$ is relatively compact in $\Aut(H)$\,.
Furthermore, $V\cap\ker\varphi$ is trivial, \,$V\ker\varphi$ closed.
\\
$\varphi(H)$ is a normal subgroup of $\overline{\varphi(H)}$ and there exists
a continous homomorphism
$\alpha\!:\overline{\varphi(H)}\to\Aut(H/\ker\varphi)$ \,(necessarily unique)
such that \,$\alpha\circ\varphi=\iota_{H/\ker\varphi}\circ\pi_1$
\,and this lifts also to an action of \,$\overline{\varphi(H)}$ on $H$\,.
\end{ThmA}\noindent
Writing $\iota_x(y) = xyx^{-1}\ (x,y\in H)$ for inner automorphisms defines
the homomorphism $\iota_H\!:H\to\Int(H)$. $\pi_1\!\!:H\to H/\ker\varphi$
denotes the quotient map. We do not exclude the case that
$\varphi(H)$ is closed. Then $V$ must be trivial.
\begin{proof}
Clearly, we can assume that
$\overline{\varphi(H)}=G$\,. First, we consider the case where $G$
is a Lie group. We follow the method of \cite{P}. Put
$H_1=H/\ker\varphi$ \,and let $\varphi_1\!\!:H_1\to G$ be the induced
homomorphism, $\pi_1\!\!:H\to H_1$ denotes the quotient mapping. Thus
$\varphi=\varphi_1\circ\pi_1$ and $\varphi_1$ is injective.
By \cite{P}\,Prop.\,1.10 there exists a
vector subgroup $V_1$ of $H_1$ \,such that 
\,$G=\varphi_1(H_1)\mspace{3mu}\overline{\varphi_1(V_1)}$\,,
\;$\varphi_1(V_1)$ is relatively compact in $G$ \,and
$\varphi_1^{-1}\bigl(\overline{\varphi_1(V_1)}\,\bigr)=V_1$ \,(hence $V_1$ is
closed). Put $V=\pi_1^{-1}(V)^0$ \,(identity component). $\pi_1^{-1}(V)$ is a
Lie group, locally isomorphic to $V$ \,(using that $\ker\varphi$ is discrete,
see \cite{V}\,2.6). It follows that
$\pi_1(V)=V_1$\,, hence $\varphi(V)=\varphi_1(V_1)$ and ($V_1$ being simply
connected) $\pi_1\vert V$ is an isomorphism. Consequently, $V$ is a
vector subgroup of $H$\,, $V\cap\ker\varphi$ is trivial,
\,$V\ker\varphi=\pi_1^{-1}(V)$
closed. $\varphi(V)$ is relatively compact.
$\varphi_1^{-1}\bigl(\overline{\varphi_1(V_1)}\,\bigr)=V_1$ is
equivalent to $\overline{\varphi_1(V_1)}\cap\varphi_1(H_1)=\varphi_1(V_1)$ and
this is also equivalent to the corresponding property for $V$\,.
$\varphi_1(H_1)$ is normal in $G$ by \cite{P}\,1.2. The
construction of a homomorphism $\alpha\!:G\to\overline{\Int(H_1)}$ is
explained in \cite{P}\,p.\,636 (using the notation $A_\varphi$ and
$\operatorname{Ad_H}$ for our $\iota_H$). $H$ and
$H_1$ are locally isomorphic, hence as explained in \cite{P}\,p.\,634,
$\overline{\Int(H)}$ and $\overline{\Int(H_1)}$ are both isomorphic to the
same subgroup of automorphisms of the Lie algebra of $H$ and it follows that
$\alpha$ lifts to a homomorphism $\alpha^+\!:G\to\overline{\Int(H)}$ with
$\alpha^+\circ\varphi=\iota_H$\,, in particular
\,$\overline{\iota_H(V)}=\alpha^+\bigl(\overline{\varphi(V)}\bigr)$
is compact. 
\\[0mm plus 1mm]
In the general case, $G$ is a connected locally compact group, hence a
projective limit of Lie groups. Let $K$ be a compact normal subgroup of $G$
such that
$G/K$ is a Lie group and the dimension of the Lie group
$D=\varphi^{-1}(K)$
becomes minimal. If $D$ would be non--discrete, take
$x\in D^0\setminus\ker\varphi$. Then $\varphi(x)\neq e$ and there would exist
a compact normal subgroup $K_1\subseteq K$ of $G$ such that
$G/K_1$ is a Lie group and $\varphi(x)\notin K_1$\,. Then
$x\notin \varphi^{-1}(K_1)$ contradicting the minimality condition for
$K$\,. Thus $D$ must be discrete. Let $\pi_K\!\!:G\to G/K$ be the quotient
mapping and put $\varphi_K=\pi_K\circ\varphi$\,. Then $\ker\varphi_K=D$ \,and
we can apply the special case to $G/K\,,\varphi_K$\,. Let $V$ be a vector
subgroup of $H$ satisfying the properties of Theorem\,\ref{dense} for
$\varphi_K$\,. As mentioned above, this implies that
$\varphi_K^{-1}(\overline{\varphi_K(V)})=V$ and then the same property for
$\varphi$ follows easily. Compactness of $K$ implies that $\varphi(V)$ is
relatively compact in $G$\,.
For $T=\overline{\varphi(V)}\,K$ we get that $G=\varphi(H)\,T$\,.
Then by \cite{P}\,L.\,1.1(ii), $\overline{\varphi(\varphi^{-1}(T))}=T$\,.
We have $\pi_K(T)\subseteq\overline{\varphi_K(V)}$ and this implies
$\varphi^{-1}(T)\subseteq\varphi_K^{-1}\bigl(\overline{\varphi_K(V)}\,\bigr)
=V$ hence $T=\overline{\varphi(V)}$\,. Let
$\alpha_K\!:G/K\to\overline{\Int(H/\ker\varphi_K)}$ be the homomorphism with 
$\alpha_K\circ\varphi_K=\iota_{H/\ker\varphi_K}$\,. Then
$\alpha^-=\alpha_K\circ\pi_K\!:G\to\overline{\Int(H/\ker\varphi_K)}$ and
as above this lifts to $\alpha\!:G\to\overline{\Int(H/\ker\varphi)}$
and $\alpha^+\!:G\to\overline{\Int(H)}$ with the
required properties. In particular
\,$\overline{\iota_H(V)}=\alpha^+\bigl(\overline{\varphi(V)}\bigr)$
is compact. For $y\in G\,,x\in H$ it follows by approximation that
$y\,\varphi(x)\,y^{-1}=\varphi\bigl(\alpha^+(y)(x)\bigr)$ hence $\varphi(H)$ is
normal in $G$\,.
\end{proof}
\begin{Pro}    \label{densepr}
{\rm (a)} Let $H_1\,,\Cal K$ be locally compact groups, $V_1$ a closed
subgroup of $H_1$\,,
$\varphi_0\!\!:V_1\to\Cal K\,,\ \alpha_0\!\!:\Cal K\to\Aut(H_1)$
continuous homomorphisms such that
$\alpha_0\circ\varphi_0=\iota_{H_1}\vert V_1\,,\
\Cal K=\overline{\varphi_0(V_1)}$\,.
Put $D=\{x^{-1}\varphi_0(x): x\in V_1\}$
(in $H_1\underset{^{\alpha_0}}{\rtimes}\Cal K$),
$G_1=(H_1\underset{^{\alpha_0}}{\rtimes}\Cal K)/D$\,. Let
$\pi\!:H_1\underset{^{\alpha_0}}{\rtimes}\Cal K\to G_1$ be the quotient map,
$\varphi_1=\pi\vert H_1$.\\
Then $\overline{\varphi_1(V_1)}=\pi(\Cal K)\cong\Cal K$ (topologically),
$\varphi_1(H_1)\cap\overline{\varphi_1(V_1)}=\varphi_1(V_1)=
\pi(\im \varphi_0)$,
\\[0mm plus 1mm]
$\varphi_1\!:H_1\to G_1\,,\ \ker\varphi_1=\ker\varphi_0\,,\
\overline{\varphi_1(H_1)}=G_1=\varphi_1(H_1)\,\overline{\varphi_1(V_1)}$ \
(``attaching $\Cal K$ to $H_1$ along $V_1$'').
\\[1mm]
{\rm (b)} Conversely, assume $H,\varphi,V,\alpha$ are as in
Theorem\,\ref{dense} with
$G=\overline{\varphi(H)}\,,\ \Cal K=\overline{\varphi(V)}$\,. \ Put
$H_1=H/\ker\varphi$ \,and let $\varphi^-\!\!:H_1\to G$ be the induced
homomorphism, $\pi_1\!\!:H\to H_1$ \,the quotient mapping, $V_1=\pi_1(V),\
\varphi_0=\varphi^-\vert V_1\,,\ \alpha_0=\alpha\vert\Cal K$\,.
Then $x\,k\mapsto\varphi(x)\,k\ (x\in H_1\,,\,k\in\Cal K)$ induces a
(topological) isomorphism of
\,$G_1=(H_1\underset{^{\alpha_0}}{\rtimes}\Cal K)/D$ and $G$\,. With
$\varphi_1^+=\varphi_1\circ\pi_1\!:H\to G_1$ it follows that $(G,\varphi)$
and $(G_1,\varphi_1^+)$ are equivalent in the sense corresponding to
\cite{P}\,Def.\,2.1.
\\[1mm plus .5mm]
If there exists a closed normal subgroup $H_2$ of $H_1$ such that
$H_1=H_2\rtimes V_1$ splits (equivalently, there exists a closed normal
subgroup $H_2^+$ of $H$ such that $H_2^+\supseteq\ker\varphi$ and
$H=H_2^+\rtimes V$ splits), then $G_1\cong H_2\rtimes\Cal K$ (restricting
the action $\alpha_0$ to $H_2$).
\end{Pro}\noindent
As usual, we identify $H,\Cal K$ with the corresponding subgroups of the
semidirect product.
\begin{proof} Using density of $\varphi_0(V)$ in $\Cal K$ one can show that
$D$ is a normal subgroup (if $V_1$ is abelian then $D$ is even central).
The properties of $G_1$ follow by routine computations (see also
\cite{P}\,p.\,637f. for a related construction).
In the situation of Theorem\,\ref{dense}, $H$ is connected, $\Cal K$ compact
and it follows that $G_1$ is $\sigma$--compact. By \cite{HR}\,Th.\,5.29
the isomorphism between $G_1$ and $G$ is also a homeomorphism. A condition
for the splitting case is \cite{P}\,Th.\,1.12 and \cite{P}\,1.12 an example
of non--splitting.
\end{proof}
\begin{Rems}
(a) A related result was shown in \cite{G2}\,Th.\,3 for the case where
$G$ is an arbitrary Hausdorff topological group. But then
$\varphi(V)$ need not be relatively compact in $G$\,. \cite{P} gives a
more direct argument for the case where $G$ is a Lie group and adds various
refinements, using the notion (from \cite{G2}) of a
``generalized maximal torus'' in $H$ \,.
This can also be extended to the case of locally compact $G$\,,
similarly as in the proof of Theorem\,\ref{dense} (but we do not need it for
our applications in Section\,\ref{mini}).
\item[(b)] In \cite{PW} extensions to \cite{G2} are given for the case where
$H$ is locally compact connected group. One can also extend most of
Theorem\,\ref{dense} to the case where
$H,G$ are locally compact groups, $H$ connected,
$\varphi\!:H\to G$ \,a continuous homomorphism such that $\ker\varphi$ is
totally disconnected, except for the property
$\overline{\varphi(V)}\cap\varphi(H)=\varphi(V)$\,.
\item[(c)]
For a connected Lie group $H$ and a discrete central subgroup $D_1$ put
$H_1=H/D_1$\,. Given a closed vector subgroup $V_1$ of $H_1$ such that
$\iota_{H_1}(V_1)$ is relatively compact in $\Aut(H_1)$ the possible choices
of compact groups $\Cal K$ and continuous homomorphisms
$\varphi_0\!\!:V_1\to\Cal K\,,\ \alpha_0\!\!:\Cal K\to\Aut(H_1)$ with
$\Cal K=\overline{\varphi_0(V_1)}$ \,(hence
$\Cal K$~must be abelian),
$\alpha_0\circ\varphi_0=\iota_{H_1}\vert V_1$
can be obtained by duality. \;Put $\Gamma=\widehat{\Cal K}\,,\,
\Gamma_0=\bigl(\overline{\iota_H(V_1)}\bigr)\sphat$ (dual groups),
$\iota_0=\iota_{H_1}\vert V_1$\,. The dual homomorphisms
$\widehat\varphi_0\!:\Gamma\to\widehat V_1\,,\linebreak
\widehat\iota_0\!:\Gamma_0\to\widehat V_1$ (called the adjoint homomorphism
in \cite{HR}\,24.37) must be injective and satisfy
$\widehat\varphi_0\circ\widehat\alpha_0=\widehat\iota_0$\,. Hence, up to
isomorphism, the discrete group $\widehat{\Cal K}$ must be a subgroup
(algebraically) of $\widehat V_1\ (\cong\R^n)$ containing
$\widehat\iota_0(\Gamma_0)$. Conversely, for any such subgroup
$\widehat{\Cal K}$ of $\widehat V_1$ (equipped with discrete topology),
taking $\widehat\varphi_0$ the inclusion, one can apply
Proposition\,\ref{dense}, producing a continuous homomorphism
$\varphi_1\!:H_1\to G_1$ with
$\overline{\varphi_1(H_1)}=G_1=\varphi_1(H_1)\,\overline{\varphi_1(V_1)}$\,,
$\overline{\varphi_1(V_1)}\cong\Cal K$ compact.
Composition with the quotient mapping $\pi_1\!\!:H\to H_1$ gives 
$\varphi=\varphi_1\circ\pi_1\!:H\to G_1$\,.
The maximal case is $\Gamma=V_1$ (discrete topology) giving $\Cal K=b(V_1)$
(Bohr compactification). By \cite{HR}\,24.41, $\varphi_0$
is injective iff $\widehat{\Cal K}$ is dense in $\widehat V_1$ and then
$\ker\varphi=D_1$ giving (up to isomorphism all) examples for
Theorem\,\ref{dense}. If $\varphi_0$ is not injective then $\ker\varphi_0$
contributes to $\ker\varphi$ and this does not fit with the setting of
Theorem\,\ref{dense} (if $\ker\varphi_0$ is non--discrete one even
gets cases where $\ker\varphi$ is non--discrete, see also Remark\,(b) in
Section\,\ref{mini} on $H=\He$).
\\
It is easy to see that if $V$ satisfies the properties of
Theorem\,\ref{dense}, the same is true for $xVx^{-1},\ x\in H$\,.
\\
The saturation condition
$\overline{\varphi(V)}\cap\varphi(H)=\varphi(V)$
\;(going back to \cite{P}) is essential for the representation in
Proposition\,\ref{densepr}(b) as a quotient of a semidirect product.
Often there exist proper subgroups $V_0$ of $V$ with
$\overline{\varphi(V_0)}=\overline{\varphi(V)}$. For example, if $G$ is second
countable there always exists a one dimensional subgroup with this property
(\cite{G1}\,Prop.\,12). But one cannot directly apply
Proposition\,\ref{densepr} with $V_0$\,. The representation in \cite{P}\,p.637
(based on \cite{P}\,Th.\,2.5) is closely related to our
Proposition\,\ref{densepr}(b), but in the notation of \cite{P} $V$ may also
contain a compact part.
\\
As indicated in (b), for general locally compact connected groups $H$
(e.g. when $H$ is not arcwise
connected) there are examples where it is not possible to choose $V$ with
the saturation condition. Then the classification of the extensions $G$
gets more involved.
\end{Rems}
\section{Fourier transform, inversion and extensions} \label{AppenB}\noindent
\cite{KL} does not touch the topic of the non--commutative Fourier transform.
We will recall here some basic definitions and review the general
background. See also \cite{F} for a slightly more detailed survey.
\\
Let $G$ be a locally compact group. We take a fixed (left) Haar measure, in
integrals we write simply $dx$ etc. $L^p(G)$ denotes the standard (complex)
$L^p$--spaces, $M(G)$ the space of bounded complex Radon measures,
convolution is denoted by $\star$\, (see \cite{HR} for basic properties).
Representations $\pi$ of $G$ will always be strongly continuous and act
unitarily on
Hilbert spaces with inner product denoted by $(\,\vert\,)$. The extension to
the algebra $L^1(G)$,
$\pi(f)=\int_G f(x)\pi(x)dx\ (f\in L^1(G)\,)$ is denoted by the same letter,
similarly for $M(G)$. The Fourier--Stieltjes algebra $B(G)$ consists of all
coefficient functions $x\mapsto(\pi(x)f\vert g)$ of all representations of
$G$\,. The Fourier algebra $A(G)$ consists of the coefficients of the left
(or right) regular representation of $G$ on $L^2(G)$ (see \cite{E}). Both
are Banach algebras under pointwise multiplication. $\VN(G)$ denotes  the
von Neumann algebra on $L^2(G)$ generated by the left regular representation.
It is isometrically isomorphic to the dual space $A(G)'$.
\\
For $\Cal H$ a Hilbert space, $\Cal B(\Cal H)$ denotes the space of
bounded linear operators, $\Cal K(\Cal H)$ the subspace of compact operators,
$\Cal S_2(\Cal H)$ the Hilbert--Schmidt operators
(with norm $\lVert\cdot\rVert_2$) and $\Cal S_1(\Cal H)$ the trace class
(norm $\lVert\cdot\rVert_1$).
\\
$\widehat G$ (the dual object of $G$) consists of the equivalence classes of
irreducible representations of $G$\,, it has a topology and a Borel structure
(\cite{Di}\,Ch.\,18). If $G$ is type~I and second countable there exists
a Plancherel measure on $\widehat G$ which we denote by $\nu$\,. There
exists a measurable selection of representatives from the equivalence classes.
We fix such a selection (defined up to a set of $\nu$--measure $0$) and will
use for simplicitly the same letter $\pi$ for a representation on the
space $\Cal H_\pi$ selected in this way from the class $\pi$\,. This gives
a ($\nu$-)\,measurable field of Hilbert spaces $(\Cal H_\pi)$
providing a decomposition of the left regular representation as a direct
integral of factorial representations, $\VN(G)\cong
\int_{\widehat G}^\oplus \Cal B(\Cal H_\pi)\,d\nu(\pi)$\,.
Following \cite{F} we will use the abbrevation
$\Cal B_1^\oplus(\widehat G,\nu)$ for the space of measurable operator
fields $(T_\pi)$ satisfying $T_\pi\in\Cal S_1(\Cal H_\pi)\ \nu
\text{-a.e. and }
\lVert(T_\pi)\rVert_1=\int_0^\infty\lVert T_\pi\rVert_1\,d\nu(\pi)<\infty\}$
(with the usual equivalence classes with respect to $\nu$\,; in the notation
of \cite{T1}\,p.\,285 this is
$\int_{\widehat G}^\oplus \Cal S_1(\Cal H_\pi)\,d\nu(\pi)$ describing
the predual of $\VN(G)$\,). Similarly $\Cal B_2^\oplus(\widehat G,\nu)$
denotes the direct integral of Hilbert spaces
$\int_{\widehat G}^\oplus \Cal S_2(\Cal H_\pi)\,d\nu(\pi)$ \
(i.e.
$\lVert(T_\pi)\rVert_2^2=
\int_{\widehat G}\lVert T_\pi\rVert_2^2\,d\nu(\pi)$\,).
\\[0mm plus 1mm]
When $G$ is abelian, $\widehat G$ coincides with the dual group of $G$\,,
$\nu$ is a (suitably normalized -- see below) left Haar measure,
$\Cal B_1^\oplus(\widehat G,\nu)=L^1(\widehat G,\nu),\
\Cal B_2^\oplus(\widehat G,\nu)=L^2(\widehat G,\nu)$. If $G$ is a compact
group, $\widehat G$ is discrete, $\Cal B_1^\oplus(\widehat G,\nu)$ is a
weighted $l^1$-sum, $\Cal B_2^\oplus(\widehat G,\nu)$ a
weighted $l^2$-sum. In the examples of the Sections\,\ref{axb},\ref{heis},%
\ref{euclm}, we have that
$\Cal H_\pi=\Cal H$ is a fixed Hilbert space for $\nu$--almost all $\pi$\,.
Then \;$\Cal B_1^\oplus(\widehat G,\nu)=L^1(\widehat G,\nu,\Cal S_1(\Cal H))$
\;(Bochner integral for $\Cal S_1(\Cal H)$--valued functions,
\cite{T1}\,Ch.\,IV,\,Sec.\,7 uses the notation
\,$L^1_{\Cal S_1(\Cal H)}(\widehat G,\nu)$\,),
\,$\Cal B_2^\oplus(\widehat G,\nu)=L^2(\widehat G,\nu,\Cal S_2(\Cal H))$.
Similarly, $\VN(G)\cong L^\infty(\widehat G,\nu,\Cal B(\Cal H))\cong
L^\infty(\widehat G,\nu)\overline\otimes\Cal B(\Cal H)$ when using
a slightly weaker notion of measurability for $\Cal B(\Cal H)$-valued
functions (\cite{T1}\,Ch.IV,\,Def.\,7.7). A slightly different example is in
Section\,\ref{heisK}.
\\[1mm plus 1.5mm]
For $f\in L^1(G)$ we write $\Cal F(f)=(\pi(f))$  for the (non--commutative)
Fourier transform of $f$ (similarly for $\mu\in M(G)$\,).
Let $\lambda(f)(g)=f\star g$ ($g\in L^2(G)$) be the left convolution operator
on $L^2(G)$. Then $\lambda$ defines an embedding of $L^1(G)$ to a w*-dense
subalgebra of $\VN(G)$ and $\Cal F$ extends to a w*-continuous isomorphism
of von Neumann algebras $\VN(G)\to L^\infty(\widehat G,\nu)$  \;(again
denoted by $\Cal F$).
\\[0mm plus 1mm]
Assume now that $G$ is unimodular. Then the crucial property of the
Plancherel measure is that for $f\in L^1(G)\cap L^2(G)$ we have
$\Cal F(f)\in\Cal B_2^\oplus(\widehat G,\nu)$ and
$\lVert f\rVert_2=\lVert\Cal F(f)\rVert_2$ \,(and this determines $\nu$
uniquely when fixing a Haar measure on $G$).\linebreak $\Cal F$~extends to an
isometric isomorphism of Hilbert spaces
$L^2(G)\to\Cal B_2^\oplus(\widehat G,\nu)$ \;(again
denoted by $\Cal F$\,; \cite{F}\,Th.\,3.31 uses the notation $\Cal P$),
it follows also that
$(f\vert g)=\int_{\widehat G}\tr(\pi(g)^*\pi(f))\,d\nu(\pi)$ 
for $f,g\in L^1(G)\cap L^2(G)$ \,($\tr(\cdot)$ denoting the trace of the
operator). For $f\in L^2(G)\cap\VN(G)$  (i.e. the functions defining a
convolution operator on $L^2(G)$; called the Hilbert algebra of $G$\,,
\cite{Di}\,13.10) the two extensions of $\Cal F$ coincide. For 
$f\in L^1(G)\cap A(G)$ we have
$\Cal F(f)\in\Cal B_1^\oplus(\widehat G,\nu)$ and
$\lVert f\rVert_A=\lVert\Cal F(f)\rVert_1$ \,(note that $\lVert\cdot\rVert_A$
does not depend Haar measure, but $\Cal F(f)$ does and this is compensated
by the Plancherel measure). $\Cal F$ extends to an
isometric isomorphism of Banach spaces
$A(G)\to\Cal B_1^\oplus(\widehat G,\nu)$ \;(again denoted by $\Cal F$\,).
For $f\in L^2(G)\cap A(G)$  the two extensions of $\Cal F$ coincide
(recall also that for unimodular $G\,,\ \VN(G)\cap A(G)\subseteq L^2(G)$\,).
The inverse mapping, for $T=(T_\pi)\in\Cal B_1^\oplus(\widehat G,\nu)\,,\
u=\Cal F^{-1}(T)\in A(G)$ is given by
$u(x)=\int_{\widehat G}\tr(\pi(x)^*\,T_\pi)\,d\nu(\pi)\ \ (x\in G)$\,.
\\[0mm plus .5mm]
In the general case (if $G$ is not necessarily unimodular), let $\Delta$
be the modular function for $G$\,. In \cite{DM}\,Th.\,5 it is shown that
there exists a measurable family $(K_\pi)$ of unbounded (densely defined)
operators (called Duflo--Moore operators) such that for $\nu$--almost all
$\pi\,,\ K_\pi$ is positive definite,
self-adjoint (on $\Cal H_\pi$) and satisfies the semi-invariance relations
\,$\pi(x)^*K_\pi\,\pi(x)=\Delta(x)K_\pi$ for all $x\in G$
\,(\cite{F}\,Th.\,3.48
uses $C_\pi=K_\pi^{-\frac12}$). This gives rise to a modified Fourier
transform
$\Cal F_2(f)=([\pi(f)K_\pi^{\frac12}])\ \ (\,[\cdot]$ denoting the closure of
the operator -- if it exists). For $f\in L^1(G)\cap L^2(G)$ one has
$\Cal F_2(f)\in\Cal B_2^\oplus(\widehat G,\nu)$ and
$\lVert f\rVert_2=\lVert\Cal F_2(f)\rVert_2$\,,
\;$\Cal F_2$ extends to an isometric isomorphism of Hilbert spaces
$L^2(G)\to\Cal B_2^\oplus(\widehat G,\nu)$ \;(again
denoted by $\Cal F_2$\,). The modification \,$\Cal F_A(f)=([\pi(f)K_\pi])$
\,satisfies for $f\in L^1(G)\cap A(G)$ that
\,$\Cal F_A(f)\in\Cal B_1^\oplus(\widehat G,\nu)$ and
\,$\lVert f\rVert_A=\lVert\Cal F_A(f)\rVert_1$\,,%
\;$\Cal F_A$ extends to an
isometric isomorphism of Banach spaces
$A(G)\to\Cal B_1^\oplus(\widehat G,\nu)$ \;(again
denoted by $\Cal F_A$\,). Similar to the unimodular case, the extensions
are compatible. For example, when\linebreak $f\in L^2(G)\cap\VN(G)$ (the left
bounded elements of $L^2$) one has
$\Cal F_2(f)_\pi=[\Cal F(f)_\pi K_\pi^{\frac12}]$ $\nu$--a.e.
The inverse mapping, for
\,$T=(T_\pi)\!\in\Cal B_1^\oplus(\widehat G,\nu)\,,\
u=\Cal F_A^{-1}(T)\in A(G)$ is again given by
$u(x)=\int_{\widehat G}\tr(\pi(x)^*\,T_\pi)\,d\nu(\pi)\ \ (x\in G)$\,.
The Plancherel measure $\nu$ depends in the general case also on the choice
of the family $(K_\pi)$\,, different Plancherel measures are equivalent
(as noted in \cite{F}\,Th.\,4.4, the formula for
$\Cal F_A^{-1}$ does not involve the Duflo--Moore operators). The formulas
given before for the unimodular
case amount to take for $K_\pi$ identity operators.
\\[0mm plus 1mm]
The isomorphism $\Cal F_A$ (resp. $\Cal F$) can be used to transfer
pointwise multiplication on $A(G)$ to define a multiplication $\sharp$ on
$\Cal B_1^\oplus(\widehat G,\nu)$ called {\it dual convolution}.
Thus for $S,T\in \Cal B_1^\oplus(\widehat G,\nu),\
\;S\sharp T=\Cal F_A(\,\Cal F_A^{-1}(S)\,\Cal F_A^{-1}(T)\,)$.
\\[0mm plus 1mm]
For a Lie group $G$ one can also extend a representation $\pi$
to the algebra of distributions of compact support. Then $\pi(d)$ will
in general be an unbounded operator, defined e.g. on the G\aa rding
subspace of $\Cal H_\pi$\,. This applies in particular to the elements of
the Lie algebra of $G$ which can be seen as distributions
supported by $\{e\}$ \,(see \cite{Wa}\;Sec.\,4.4.1).\vspace{3mm}

\noindent
We want to add here
some comments about a statement of \cite{F}. In proving
results on pointwise inversion of the Plancherel transform F\"uhr asserts
about a "gap" in a
density argument used by earlier authors and introduces a new method based on
positivity and bounded vectors (see also \cite{F} Rem.\,4.16(a)).
The author seems to be unaware of Stinespring's paper \cite{S} (in particular
Th.\,9.17 from there and its proof).  Terminology and notation of \cite{S}
is quite different. For the unimodular case, we give here
an argument for the (critical) "only if"-part of \cite{F} Th.\,4.15, in the
spirit of  \cite{S}\,Rem.\,9.4, adapted to our notation above
(which is close to  \cite{F}).
The general (non--unimodular) case is sketched in  Remark (b) below.
\\[1mm plus 1mm]
For clarity, we denote here the extension of $\Cal F$ to $L^2(G)$ by
$\Cal P$ (Plancherel transform). For $A\in B_1^\oplus,\ a=\Cal F_A^{-1}(A)$,
$G$ unimodular (and type I), we want to
show that $a\in L^2(G)$ implies $A\in B_2^\oplus$ and $A=\Cal P(a)$.
\\[0mm plus 1mm]
First, we treat the special case where, in addition, $\Cal P(a)\in B_1^\oplus$
holds. Put $A^{(0)}=A-\Cal P(a),\ a_0=\Cal F_A^{-1}(A^{(0)})$.
Then $A^{(0)}\in B_1^\oplus$ and for $g\in (L^1\cap L^2)(G)$ an easy
computation, using the Parseval equality (\cite{F}\,Th.\,4.4) and 
unitarity of  $\Cal P$ (\cite{F}\,Th.\,3.31), shows that
$\int_{\widehat G}\tr(\pi(g)\,A^{\!(0)\,*}_\pi)\,d\nu(\pi)=0$.
\\[0mm plus 1mm]
Recall that $\Cal F$ extends to a w*-continuous isomorphism of  the
von Neumann algebras $\VN(G)$ and
$\int_{\widehat G}^\oplus \Cal B(\Cal H_\pi)\,d\nu(\pi)$. $B_1^\oplus$ is the
predual of
$\int_{\widehat G}^\oplus \Cal B(\Cal H_\pi)\:d\nu(\pi)$
\,(\cite{T1}\,Ch.IV Prop.\,8.34) and the Parseval equality of
\cite{F}\,Th.\,4.4
says that $\Cal F_A^{-1}$ is the dual mapping for this isomorphism
(for the sesquilinear duality).
It follows that $\{\Cal F(g):\;g\in (L^1\cap L^2)(G)\}$ is w*-dense in
$\int_{\widehat G}^\oplus \Cal B(\Cal H_\pi)\:d\nu(\pi)$, hence
$A^{(0)}=0$, i.e., $A_\pi=\Cal P(a)_\pi$ holds $\nu$-a.e.
\\
For the general case, take $h\in L^1(G)$. Then by Fubini's theorem
(see \cite{F}\,p.\,81) and continuity of $\Cal P$, we have
$\Cal P(h\star a)=\Cal F(h)\circ\Cal P(a)$ and similarly,
$\Cal F_A^{-1}(\Cal F(h)\circ A)=h\star a$\,.
If $h\in (L^1\cap L^2)(G)$ then $\Cal F(h)=\Cal P(h)\in B_2^\oplus$,
hence $\Cal F(h)\circ\Cal P(a)\in B_1^\oplus$. Thus $\Cal F(h)\circ A$
satisfies the additional requirement of the special case and it follows
that $\pi(h)\circ A_\pi=\pi(h)\circ\Cal P(a)_\pi$ holds $\nu$-a.e.
(with an exceptional set depending possibly on $h$).
\\
For $c>0$, put $M_c=\{\pi\in\widehat G:\;\lVert A_\pi\rVert_1\le c\}$
and $A^{(c)}_\pi=A_\pi-\Cal P(a)_\pi$ for $\pi\in M_c$\,,
$A^{(c)}_\pi=0$ else,
$A^{(c)}=(A^{(c)}_\pi)$. Then
$\lVert A_\pi\rVert_2\le\lVert A_\pi\rVert_1$ implies
$A^{(c)}\in B_2^\oplus$ and we get  
$\int_{\widehat G}\tr(\pi(h)\,A^{(c)}_\pi)\,d\nu(\pi)=0$
for all $h\in (L^1\cap L^2)(G)$. Since
$\{\Cal P(h):\;h\in (L^1\cap L^2)(G)\}$ is dense in $B_2^\oplus$, it follows
that $A^{(c)}=0$ in $B_2^\oplus$, thus  $A_\pi=\Cal P(a)_\pi$ a.e. on
$M_c$\,. $c>0$ was arbitrary, hence we arrive at $A=\Cal P(a)\in B_2^\oplus$.
\\
A similar argument shows that just assuming $a\in\VN(G)$ already implies
$a\in L^2(G)$. Also one gets that $\Cal P(Bg)=\Cal F(B)\circ\Cal P(g)$
for $B\in\VN(G),\;g\in L^2(G)$ and then that $\Cal P$ coincides with
$\Cal F$ on $\VN(G)\cap L^2(G)$.

\begin{Rems}
(a) Some care is needed when using the duality between $A(G)$ and $\VN(G)$.
Traditionally it is obtained from the extension of the pairing \;
$\langle f,g\rangle\;=\linebreak
\int_Gf(x)\,g(x)\,dx\ \;(f\in A(G),\;g\in L^1(G))$ \,(and then the dual
$A(G)$-module structure on $\VN(G)$ becomes the extension of pointwise
multiplication). Under Fourier transform there is an extended Parseval
equality \
$\int_Gf(x)\,\overline{g(x)}\,dx=\\
\int_{\widehat G}\tr\bigl(\pi(g)^*\,\Cal F_A(f)_\pi\bigr)\,d\nu(\pi)$. To get
expressions for $\langle f,g\rangle$ (and more generally for
$\langle f,T\rangle$ where $T\in\VN(G)$\,) one has to analyze the effect of
the conjugation $g\mapsto\bar g$ (which extends also to a semi-linear
automorphism of $\VN(G)$\,) on the disintegration (see also \cite{F}\,p.110).
\item{(b)} \,The density argument above works similarly (with some more
technicalities)
for non-unimodu\-lar groups $G$ (of type I and second countable). Let
$K=(K_\pi)$ be the field of Duflo-Moore operators, and similar to \cite{F},
\,$[T]=([T_\pi])$ shall denote the field built from the closures of
a field of densely defined operators $T=(T_\pi)$. Then for $A,a$ as above,
the aim is to show
that $a\in L^2(G)$ implies $[A\circ K^{-\frac12}]=\Cal F_2(a)\in B_2^\oplus$.
If in
addition $[\Cal F_2(a)\circ K^{\frac12}]\in B_1^\oplus$ holds, this can be done as in
the unimodular case above. For $u\in L^1(G)$ one has
$\Cal F_A^{-1}(A\circ\Cal F(u))=a\star(\Delta\,u)$.
Then observe that for those $\pi$ where the
semi-invariance relations hold (thus for $\nu$-almost all $\pi$) one
gets for $(1+\Delta^{-\frac12})\,f\in L^1(G)$ that
\,$\pi(f)\circ K^{\frac12}_\pi\subseteq
K^{\frac12}_\pi\pi(\Delta^{-\frac12}f)$  (in the sense of unbounded operators,
i.e., the right operator is an [in general proper] extension of the left one).
It follows also that
$[K^{-\frac12}\circ\Cal F_2(f)\circ K^{\frac12}]=\Cal F_2(\Delta^{-\frac12}f)$ \,for 
$(1+\Delta^{-\frac12})\,f\in L^2(G)$ \,(compare \cite{F}\,L.\,4.14).
This implies
$\Cal F_2(f)\circ\Cal F_2(v)=
[\Cal F_2(f\star(\Delta^\frac12v))\circ K^{\frac12}]$ for
$f\in L^2(G), v\in(L^1\cap L^2)(G)$ and using this, one can treat the general
case similar as above.
\item{(c)} \,The Plancherel theorem was shown by Duflo--Moore \cite{DM} for an
arbitrary second countable group $G$ for which the left regular representation
$\lambda$ is of type I (in particular, for all second countable groups of type
I). The restrictions in the earlier work of Tatsuuma, Kleppner, Lipsman
are unnecessesary: if $\lambda$ is of type I, then the same is true for the
regular representation of $N=\ker\Delta$ \,(\cite{DM} p.240);\linebreak
"$N$ regularly embedded" (used also in \cite{F}) can be replaced by a weaker
property which as
they show is always satsfied when $\lambda$ is of type I\,.
\item{(d)} \,The  Plancherel theorem (and the consequences above) extends
to $\sigma$-compact $G$ (instead of second countable): \;for
$\fr K=
\{K: \text{compact normal subgroup of } G \text{ with }\linebreak
G/K \text{ second countable}\}$ the Kakutani-Kodaira theorem (see
\cite{LN}\,L.\,24) says that
$G$ is the projective limit of $(G/K)_{K\in\fr K}$\,. This
implies that $\widehat G=\bigcup_{K\in\fr K}\widehat{G/K}$  (more precisely,
$\widehat{G/K}$ is homeomorphic to an open-closed subset of $\widehat G$).
Similarly to \cite{LN}\,L.\,1 there are embeddings of $L^2(G/K)$ into
$L^2(G)$
and $L^2(G)=\bigcup_{K\in\fr K}L^2(G/K)$. $\VN(G/K)$ is isomorphic to a
von Neumann subalgebra of $\VN(G)$ and \;$\bigcup_{K\in\fr K}\VN(G/K)$ is
w*-dense in $\VN(G)$. $A(G/K)$ is isomorphic to a subalgebra of $A(G)$ and
\;$A(G)=\bigcup_{K\in\fr K}A(G/K)$ \,(see \cite{KL}\,Prop.\,2.4.2).
When $G$ is not second countable
(but $\sigma$-compact) the Plancherel measure is no longer $\sigma$-finite.
\item{(e)} \,In \cite{S} a duality theory (including Plancherel theorem and
Fourier inversion) is developed for arbitrary unimodular locally compact
groups (with no topological countability assumptions and no restrictions on the
type). This extends to the non-unimodular case using the Plancherel weight
(\cite{T2} p.\,67). It gives rather abstract statements in terms of
unbounded operators on $L^2(G)$. If the group is second countable (or just
$\sigma$-compact), there is a disintegration of the weight built up on
the central disintegration of the left regular representation. Explicit
examples can be found in a paper of Kajiwara (ref.[69] in \cite{F}). In
the non-type I case, non-trivial subfactors of $ \Cal B(\Cal H_\sigma)$
(and non-trivial weights) arise. In the examples these are often related
to ergodic theory.
\end{Rems}


\begin{thebibliography}{1}
\bibitem[BH]{BH}
W.R.~Bloom, H.~Heyer,
\emph{Harmonic Analysis of Probability Measures on Hypergroups},
 de~Gruyter, Berlin, New York, 1995.
\vskip0,2cm
\bibitem[CG1]{CG}
Y.~Choi, M.~Ghandehari,
\emph{Weak and cyclic amenability for Fourier algebras of connected Lie
groups},
J. Funct. Anal. \textbf{266} (2014), 6501--6530.
\vskip0,2cm
\bibitem[CG2]{CG2}
Y.~Choi, M.~Ghandehari,
\emph{Weak amenability for Fourier algebras of 1-connected nilpotent Lie
groups},
J. Funct. Anal. \textbf{268} (2015), 2440--2463.
\vskip0,2cm
\bibitem[Da]{D}
H.G.~Dales, \emph{Banach algebras and automatic continuity}, 
London Math. Soc. Monographs, New Series 24, Oxford University Press,
New York, 2000.
\vskip0,2cm
\bibitem[Di]{Di}
J.~Dixmier, \emph{$C^{*}$-Algebras},
North Holland, Amsterdam--New~York--Oxford, 1977.
\vskip0,2cm
\bibitem[DM]{DM}
M.~Duflo, C.C.~Moore,
\emph{On the regular representation of a nonunimodular locally compact group},
  J. Functional Analysis \textbf{21} (1976), 209--243.
\vskip0,2cm
\bibitem[DS]{DS}
 N. Dunford, J.~T. Schwartz,  \emph{Linear Operators I}, 
 Interscience,  New~York, London, 1958.
\vskip0,3cm
\bibitem[E]{E}
P.~Eymard, \emph{L'alg\`ebre de Fourier d'une groupe localement
compact}, Bull. Soc. Math. France \textbf {92} (1964), 181--236.
\vskip0,2cm
\bibitem[Fo]{Fo}
G.B.~Folland, \emph{A course in abstract harmonic analysis},
CRC Press, Boca Raton, FL, 1995
\vskip0,3cm
\bibitem[Fu]{F}
H.~F\"uhr, \emph{Abstract harmonic analysis of continuous wavelet transforms},
Lecture Notes Math.\,1863, Springer, Berlin-Heidelberg-New York, 2005.
\vskip0,2cm
\bibitem[FL]{FL}
H.~Fujiwara, J.~Ludwig,
\emph{Harmonic analysis on exponential solvable Lie groups},
Springer Monographs in Mathematics, Springer, Tokyo, 2015.
\vskip0,3cm
\bibitem[FR]{FR}
B.~E.~Forrest, V.~Runde,
\emph{Amenability and weak amenability of the Fourier algebra},
Math. Z. \textbf{250} (2005), 731--744.
\vskip0,3cm
\bibitem[G1]{G1}
M.~Goto, \emph{Dense imbeddings of locally compact connected groups},
Ann. of Math. (2) \textbf{61} (1955), 154--169.
\vskip0,3cm
\bibitem[G2]{G2}
M.~Goto, \emph{Immersions of Lie groups},  J. Math. Soc. Japan \textbf{32}
(1980), 727--749.
\vskip0,3cm
\bibitem[HM]{HM}
K.~H.~Hofmann, S.~A.~Morris, \emph{The Structure of
Compact Groups}, Walter~de~Gruyter, Berlin--New~York,
1998.
\vskip0,3cm
\bibitem[HR]{HR}
E.~Hewitt, K.~A.~Ross, \emph{Abstract harmonic analysis I},
Springer, Berlin, 1979.
\vskip0,3cm
\bibitem[J]{J}
B.~E.~Johnson,
\emph{Non-amenability of the Fourier algebra of a compact group},
J. London Math. Soc. (2) \textbf{50} (1994), 361--374.
\vskip0,3cm
\bibitem[KL]{KL}
E.~Kaniuth, A.T.-M.\ Lau, \emph{Fourier and Fourier-Stieltjes algebras on
locally compact groups}, Mathematical Surveys and Monographs 231. American
Mathematical Society, Providence, RI, 2018.
\vskip0,3cm minus 1mm
\bibitem[Li]{Li}
J.R.~Liukkonen, \emph{Dual spaces of groups with precompact conjugacy
classes}, Trans. Amer. Math. Soc. \textbf{180} (1973), 85--108.
\vskip0,3cm
\bibitem[L]{L}
V.~Losert, \emph{On the center of group
$C^*$-algebras}, J. Reine Angew. Math. \textbf{554} (2003), 105--138.
\vskip0,3cm
\bibitem[LLSS]{LLSS}
H.~H.~Lee, J.~Ludwig, E.~Samei, N.~Spronk, \emph{Weak amenability of Fourier
algebras and local synthesis of the anti-diagonal}, Adv. Math. \textbf{292}
(2016), 11--41.
\vskip0,3cm
\bibitem[LNPS]{LN}
V.~Losert, M.~Neufang, J.~Pachl, J.~Stepr\={a}ns,
\emph{Proof of the Ghahramani--Lau conjecture}, Adv. Math. \textbf{290}
(2016), 709--738.
\vskip0,3cm
\bibitem[Pl]{Pl}
R.~J.~Plymen,
\emph{Fourier algebra of a compact Lie group},
unpublished, see  arXiv:math/0104018, 2001.
\vskip0,3cm minus 1mm
\bibitem[Po]{P}
D.~Poguntke, \emph{Dense Lie group homomorphisms},  J. Algebra \textbf{169}
(1994), 625--647. 
\vskip0,3cm
\bibitem[PW]{PW}
W.~H.~Previts, T.~S.~Wu, \emph{Immersions of locally compact connected groups},
Bull. Inst. Math. Acad. Sinica \textbf{31} (2003), 225--242. 
\vskip0,3cm
\bibitem[St]{S}
W.F.~Stinespring,
\emph{Integration theorems for gages and duality for unimodular groups},
Trans. Amer. Math. Soc. \textbf{90} (1959), 15--56.
\vskip0,3cm
\bibitem[Su]{Su}
M.~Sugiura,
\emph{Unitary representations and harmonic analysis. An introduction},
Second edition, North-Holland Publishing Co., Amsterdam 1990.
\vskip0,3cm
\bibitem[T1]{T1}
M.~Takesaki, \emph{Theory of operator algebras I}, Springer,
New~York--Heidelberg 1979.
\vskip0,3cm
\bibitem[T2]{T2}
M.~Takesaki, \emph{Theory of operator algebras II}, Springer,
Berlin 2003.
\vskip0,3cm
\bibitem[Va]{V}
V.~S.~Varadarajan, \emph{Lie Groups, Lie Algebras and
Their Representations}, Springer,
Berlin--\linebreak Heidelberg--New~York, 1984.
\vskip0,3cm
\bibitem[Vi]{Vi}
N.Ja.~Vilenkin,
\emph{Special functions and the theory of group representations},
Transl. of Math. Monographs,
American Mathematical Society, Providence, R.I. 1968.
\vskip0,3cm
\bibitem[W]{Wa}
G.~Warner, \emph{Harmonic analysis on semi--simple Lie
groups.\;I}, Springer, New\,York--Heidel\-berg, 1972.
\end{thebibliography}
\end{document}